\documentclass[twoside,11pt]{article}

\usepackage[margin=1in]{geometry}
\usepackage{mathptmx} 
\usepackage{authblk} 
\usepackage{setspace} 
\usepackage{parskip} 
\usepackage{fancyhdr} 
\usepackage{lastpage} 
\usepackage{amsmath,amssymb,amsthm}
\usepackage{bm} 
\usepackage{mathtools} 
\usepackage{physics}

\theoremstyle{plain}
\newtheorem{theorem}{Theorem}[section]
\newtheorem{lemma}[theorem]{Lemma}
\newtheorem{proposition}[theorem]{Proposition}

\theoremstyle{definition}
\newtheorem{definition}[theorem]{Definition}

\usepackage[
  backend=biber,       
  sorting=nty,          
  maxnames=3,           
  minnames=1,           
  maxbibnames=5,        
  minbibnames=3,       
  doi=true,             
  isbn=false,           
  url=false,            
  eprint=false          
]{biblatex}
\DeclareFieldFormat{journaltitle}{\textit{#1}}  
\DeclareFieldFormat{booktitle}{\textit{#1}}    
\DeclareFieldFormat{title}{#1.} 

\usepackage{lineno}

\newcommand{\keywords}[1]{\noindent\textbf{Keywords:} #1}

\usepackage[
  colorlinks=true,
  linkcolor=blue,
  citecolor=blue,
  urlcolor=blue,
  pdfauthor={Yiyuan Wang},
  pdftitle={An Iterative Computational Framework for Infinite-Horizon Mean-Field Linear-Quadratic Zero-Sum Stochastic Differential Games},
  pdfsubject={Computational Approach; Mean-Field Linear-Quadratic Stochastic Differential Games in an Infinite Horizon},
  pdfkeywords={Computational Approach; Mean-Field Linear-Quadratic Stochastic Differential Games in an Infinite Horizon}
]{hyperref}
\usepackage{doi} 
\usepackage{url}

\title{An Iterative Computational Framework for Infinite-Horizon Mean-Field Linear-Quadratic Zero-Sum Stochastic Differential Games}
\author{Yiyuan Wang \thanks{Yiyuan Wang is at the Shandong University Zhongtai Securities Institute for Financial Studies, Shandong University, 27 Shanda Nanlu, Jinan, P.R. China, 250100 (Email: wangyiyuan@mail.sdu.edu.cn).}}

\date{}

\begin{document}
\maketitle

\begin{abstract}
This work develops an iterative computational framework to obtain saddle-point solutions for infinite-horizon two-person mean-field linear-quadratic zero-sum stochastic differential games. By generalizing classical iterative framework, we construct a monotonically increasing matrix sequence to decouple the strongly coupled, analytically intractable original problem into a set of tractable subproblems. By sequentially computing the stabilizing solutions of the coupled algebraic Riccati equations for these subproblems, we can further derive the solution to the original problem. Rigorous convergence analysis is established to validate the proposed iterative algorithm. Different from existing algorithms limited to specific simplified setups, this framework proposes the first universal computational paradigm applicable to a broad class of game-driven Riccati equations.\\
\vspace{0.5em}
 \\
\end{abstract}

\section{Introduction} 
\label{sec:introduction}
 
Mean-field stochastic differential equations (MF-SDEs), also known as McKean-Vlasov SDEs, originate from the early work of \cite{Kac1956} and were rigorously established by \cite{McKean1966}. Over decades, extensive research has been conducted on the optimal control and differential games of MF-SDEs. Recent studies have further explored linear quadratic control problems, indefinite stochastic control and mean-field differential games, with rich theoretical results concerning solvability, Riccati equations and feedback representations achieved by numerous scholars; see, for example, \cite{Buckdahn-Li-Peng-Rainer2017, Sun-Yong2020a, Sun-Yong2020b, Li-Li-Yu2020, Tian-Yu-Zhang2020, Sun-Wang-Wu2021} and their reference.

This paper focuses on finding the saddle point of two-person mean-field linear-quadratic zero-sum stochastic differential games in an infinite time horizon (MF-LQZSSDGs). \cite{Li2021} proves that both the existence of open-loop and closed-loop saddle points are characterized by the solvability of a system of two coupled generalized algebraic Riccati equations with static stabilizing solutions. Notably, algebraic Riccati equations with indefinite quadratic terms are conventionally termed game-theoretic algebraic Riccati equations (GTAREs) in the literature. Early studies concerning this class of Riccati equations primarily focused on developing iterative algorithms for their special cases, as documented in \cite{Feng2010,Dragan2011,Dragan2008,Dragan201501,Dragan201502,Ivanov2015,Dragan2017,Ivanov2018,Aberkane2023,Sun2023}. More recently, \cite{wang2025} proposed a unified iterative framework to compute closed-loop saddle points for infinite-horizon two-person linear-quadratic zero-sum stochastic differential games. Building upon this iterative framework, \cite{wang2026} further developed a reinforcement learning-based solution algorithm. Nevertheless, the existing literature still lacks a comprehensive and systematic solution framework for the general coupled game-theoretic algebraic Riccati equations considered in this work, which constitutes the core research gap addressed by our study.

We develops a novel solution method for MF-LQZSSDGs by decoupling their intrinsic coupled structures. Specifically, we construct a monotonically increasing matrix sequence by extending the unified iterative framework established in \cite{wang2025}. This decomposition transforms the original intractable MF-LQZSSDGs into a sequence of tractable interconnected subproblems. By sequentially solving these subproblems via deriving the stabilizing solutions of the corresponding coupled algebraic Riccati equations (AREs), we eventually obtain the stabilizing solutions for the original MF-LQZSSDGs. The main contributions of this paper are summarized as follows:

\begin{itemize}
    \item General computational framework: Existing studies are only limited to special scenarios of the investigated problem. In contrast, this work overcomes the core challenges brought by the inherent coupling effect and proposes the first general computational approach to seek saddle points for MF-LQZSSDGs.
    \item Extended applicability: The proposed computational method breaks through the limitations of classical frameworks and is applicable to a broad class of game-theoretic algebraic Riccati equations, including deterministic GTAREs, stochastic $H_\infty$-type GTAREs, stochastic GTAREs derived from zero-sum linear-quadratic stochastic differential games, and the more general coupled AREs associated with MF-LQZSSDGs.
    \item Generalized solvability conditions: This paper establishes a set of generalized conditions, which provide explicit and effective criteria for characterizing the solvability of coupled AREs arising from MF-LQZSSDGs.
\end{itemize}

The remainder of this paper is organized as follows. Section \ref{sec:preliminary} covers the mathematical framework of MF-LQZSSDGs. Section \ref{sec:main_results} presents the main results, including the key structural properties of the problem, the criteria for the existence and uniqueness of the stabilizing solutions to the subproblems, and the proofs for the boundedness and convergence of the iterative sequence. Section \ref{sec:example} details a numerical example.

\section{Preliminary} 
\label{sec:preliminary}

\subsection{Notation}
Let us first introduce the following notation:
\begin{itemize}
    \item $\mathbb{R}^n$: $n$-dimensional real Euclidean space; $\mathbb{C}^{-}$: the set of complex numbers with negative real part; $\mathbb{R}^{n\times m}$: the space of $n\times m$ real matrices; $\mathbb{S}^n$: the set of all $n\times n$ symmetric matrices; $\overline{\mathbb{S}}_{+}^n$: the set of all $n\times n$ symmetric positive semi-definite matrices; $\mathbb{S}_{+}^n$: the set of all $n\times n$ symmetric positive definite matrices. 
    \item $I_n$: the identity matrix of size $n$; $\mathbb{O}_{n \times m}$: the null matrix of size $n \times m$. It can be simplified as 0 when no ambiguity is generated; $A_{m \times n}^{\times (l)}$ denotes the object represented by $A_{m \times n}$ arranged repeatedly for $l$ times.
    \item $A^{\top}$: the transpose of the matrix $A$; $\langle\cdot,\cdot\rangle$: the inner product on a Hilbert space. If $A\in\mathbb{S}_{+}^n$ (resp., $A\in\overline{\mathbb{S}}_{+}^n$), we write $A\succ0$ (resp., $A\succeq0$). For any $A,B\in\mathbb{S}^n$, we use the notation $A\succ B$ (resp., $A\succeq B$) to indicate that $A - B\succ0$ (resp., $A - B\succeq0$).
    \item Let $\mathbb{H}$ be a Euclidean space, and we define the following space \footnote{$\varphi \in \mathbb{F}$ denotes that $\varphi$ is $\mathbb{F}$-progressively measurable.}: 
    $L_{\mathbb{F}}^2(\mathbb{H}) = \{\varphi: [0,\infty) \times \Omega \rightarrow \mathbb{H} \mid \varphi \in \mathbb{F},\mathbb{E} \int_{0}^{\infty} |\varphi(t)|^2 \, dt < \infty \} $; $\mathcal{U}_i = L_{\mathbb{F}}^2(\mathbb{R}^{m_i})(i = 1,2)$, $\mathcal{X}_{\mathrm{loc}}[0,\infty) = \{\varphi: [0,\infty) \times \Omega \rightarrow \mathbb{R}^n \mid \varphi \in \mathbb{F} \,\text{ is continuous},\, \mathbb{E}\left[\sup_{0 \leq t \leq T} |\varphi(t)|^2\right] < \infty, \forall\, T > 0 \}$; $\mathcal{X}[0,\infty) = \{\varphi \in \mathcal{X}_{\mathrm{loc}}[0,\infty) \mid \mathbb{E} \int_{0}^{\infty} |\varphi(t)|^2 \, dt < \infty \}$.
\end{itemize}

\subsection{Mean-Field Linear-Quadratic Stochastic Differential Games in an Infinite Horizon}

Let $(\Omega,\mathcal{F},\mathbb{F},\mathbb{P})$ be a complete filtered probability space on which a r-dimensional standard Brownian motion $W = \{W(t)^{\top}=(w_{1}(t),\dotsb,w_{r}(t));t\geq0\}$ is defined with $\mathbb{F}=\{\mathcal{F}_t\}_{t\geq0}$ being the usual augmentation of the natural filtration generated generated by $W$. We consider the following controlled linear mean-field stochastic differential equation (MF-SDE, for short) on $[0,\infty)$:
\begin{equation}
\label{mflqsdg:sde}
\begin{cases}
\begin{aligned}
dX(t) &= \left\{ AX(t) + \bar{A}\mathbb{E}[X(t)] + B_1u_1(t) + \bar{B}_1\mathbb{E}[u_1(t)] + B_2u_2(t) + \bar{B}_2\mathbb{E}[u_2(t)] \right\} dt \\
& + \sum_{l=1}^{r}\left\{ C_{l}X(t) + \bar{C}_{l}\mathbb{E}[X(t)] + D_{l,1}u_1(t) + \bar{D}_{l,1}\mathbb{E}[u_1(t)] + D_{l,2}u_2(t) + \bar{D}_{l,2}\mathbb{E}[u_2(t)] \right\} dw_l(t), \quad t\ge 0, \\
X(0) &= x.
\end{aligned}
\end{cases}
\end{equation}

In the above, $X(\cdot)$ is the state process taking values in $\mathbb{R}^n$ with $x$ being the initial state. For $i=1,2$, $u_i(\cdot)$ is the control process of Player $i$, taking values in $\mathbb{R}^{m_i}$, respectively. The coefficients $A, \bar{A}, C_{l}, \bar{C}_{l} \in \mathbb{R}^{n\times n}$, $B_1, \bar{B}_1, D_{l,1}, \bar{D}_{l,1} \in \mathbb{R}^{n\times m_1}$, $B_2, \bar{B}_2, D_{l,2}, \bar{D}_{l,2} \in \mathbb{R}^{n\times m_2}(1 \leq l \leq r)$ are given constant matrices. In the subsequent discussion, we use $\hat{\Pi}$ to denote $\Pi+\bar{\Pi}$.

In this paper, Player 1 and Player 2 share the same performance functional:
\begin{equation}
\label{mflqsdg:performance_functional}
J(x;u_1(\cdot),u_2(\cdot)) = \mathbb{E}\int_0^\infty g(t,X(t),u_1(t),u_2(t),\mathbb{E}[X(t)],\mathbb{E}[u_1(t)],\mathbb{E}[u_2(t)]) dt,
\end{equation}
with
\begin{align*}
g(t,x,u_1,u_2,\bar{x},\bar{u}_1,\bar{u}_2) &= \left\langle
\begin{pmatrix}
Q & S_{1}^\top & S_{2}^\top \\
S_{1} & R_{11} & R_{12} \\
S_{2} & R_{21} & R_{22}
\end{pmatrix}
\begin{pmatrix}
x \\ u_1 \\ u_2
\end{pmatrix},
\begin{pmatrix}
x \\ u_1 \\ u_2
\end{pmatrix}
\right\rangle   + \left\langle
\begin{pmatrix}
\bar{Q} & \bar{S}_{1}^\top & \bar{S}_{2}^\top \\
\bar{S}_{1} & \bar{R}_{11} & \bar{R}_{12} \\
\bar{S}_{2} & \bar{R}_{21} & \bar{R}_{22}
\end{pmatrix}
\begin{pmatrix}
\bar{x} \\ \bar{u}_1 \\ \bar{u}_2
\end{pmatrix},
\begin{pmatrix}
\bar{x} \\ \bar{u}_1 \\ \bar{u}_2
\end{pmatrix}
\right\rangle,
\end{align*}
where
\[
\begin{cases}
Q, \bar{Q} \in \mathbb{S}^n, S_{1}, \bar{S}_{1} \in \mathbb{R}^{m_1\times n},S_{2}, \bar{S}_{2} \in \mathbb{R}^{m_2\times n}, \\
R_{11}, \bar{R}_{11} \in \mathbb{S}^{m_1},R_{22}, \bar{R}_{22} \in \mathbb{S}^{m_2}, R_{12} = R_{21}^\top, \bar{R}_{12} = \bar{R}_{21}^\top \in \mathbb{R}^{m_1\times m_2}.
\end{cases}
\]

For $(x,u_1,u_2)\in\mathbb{R}^n\times\mathcal{U}_1\times\mathcal{U}_2$, the solution $X(\cdot;x,u_1,u_2)$ to the MF-SDE in \eqref{mflqsdg:sde} may only exist in $\mathcal{X}_{loc}[0,\infty)$, which renders $J(x;u_1,u_2)$ ill-defined. We define the set of admissible controls as $\mathcal{U}_{ad}(x)=\{ (u_1,u_2)\in\mathcal{U}_1 \times \mathcal{U}_2\mid X(\cdot;x,u_1,u_2)\in\mathcal{X}[0,\infty)\}$. A pair $(u_1,u_2)\in\mathcal{U}_{ad}(x)$ is called an admissible control pair for the initial state $x$, and the corresponding $X(\cdot;x,u_1,u_2)$ is referred to as the admissible state process. In this case, $J(x, u_1, u_2)$ is clearly well-defined.

In this zero-sum mean-field game, Player 1 (\textit{the maximizer}) selects control $u_1$ to maximize \eqref{mflqsdg:performance_functional}, while Player 2 (\textit{the minimizer}) chooses $u_2$ to minimize the same function. The problem is to find an admissible control pair $(u_1^*,u_2^*)$ that both players can accept. We denote the above-mentioned problem as $({\text{MF-SDG}})^{\,0}_{\infty}$ for short. For a description of the $({\text{MF-SDG}})^{\,0}_{\infty}$ problem, refer to \cite{Li2021}. More detailed information can be found therein. 

The $({\text{MF-SDG}})^{\,0}_{\infty}$ problem corresponds to following system of coupled AREs:
\begin{equation}
\label{mflqsdg:gtare}
    \begin{cases}
        Q(P)-S(P)^{\top}R(P)^{-1}S(P)=0\\
        \hat{Q}(\hat{P},P)-\hat{S}(\hat{P},P)^{\top}\hat{R}(P)^{-1}\hat{S}(\hat{P},P)=0\\
         R_{11}(P) \prec 0,\hat{R}_{11}(P) \prec 0,R_{22}(P) \succ 0,\hat{R}_{22}(P) \succ 0 \\
    \end{cases}    
\end{equation}
where
\begin{equation}
\label{mflqsdg:gtare_auxiliarymatrix}
\begin{aligned}
    &Q(P)=PA + A^{\top}P+\sum_{l=1}^{r}C_{l}^{\top}PC_{l} + Q,\quad \hat{Q}(\hat{P},P)=\hat{P}\hat{A} + \hat{A}^{\top}\hat{P}+\sum_{l=1}^{r}\hat{C}_{l}^{\top}P\hat{C}_{l} + \hat{Q},\\
    &S(P)=\begin{bmatrix}
    S_1(P) \\
    S_2(P)
    \end{bmatrix}=\begin{bmatrix}
    B_1^{\top}P+\sum_{l=1}^{r}D_{l,1}^{\top}PC_{l}+S_1 \\
    B_2^{\top}P+\sum_{l=1}^{r}D_{l,2}^{\top}PC_{l}+S_2\\
    \end{bmatrix},\\
    &\hat{S}(\hat{P},P)=\begin{bmatrix}
    \hat{S}_1(\hat{P},P) \\
    \hat{S}_2(\hat{P},P)
    \end{bmatrix}=\begin{bmatrix}
    \hat{B}_1^{\top}\hat{P}+\sum_{l=1}^{r}\hat{D}_{l,1}^{\top}P\hat{C}_{l}+\hat{S}_1 \\
    \hat{B}_2^{\top}\hat{P}+\sum_{l=1}^{r}\hat{D}_{l,2}^{\top}P\hat{C}_{l}+\hat{S}_2\\
    \end{bmatrix},\\
    &R(P)=\begin{bmatrix}
    R_{11}(P) & R_{12}(P)\\
    R_{21}(P) & R_{22}(P)\\
    \end{bmatrix},\quad R_{ij}(P)=R_{ij}+\sum_{l=1}^{r}D_{l,i}^{\top}PD_{l,j}\,(i,j = 1,2),\\
    &\hat{R}(P)=\begin{bmatrix}
    \hat{R}_{11}(P) & \hat{R}_{12}(P)\\
    \hat{R}_{21}(P) & \hat{R}_{22}(P)\\
    \end{bmatrix},\quad \hat{R}_{ij}(P)=\hat{R}_{ij}+\sum_{l=1}^{r}\hat{D}_{l,i}^{\top}P\hat{D}_{l,j}\,(i,j = 1,2).
\end{aligned}
\end{equation}

To facilitate the subsequent description, we first define MF-SDE \eqref{mflqsdg:sde} as system 
\[
[A, \bar{A}, \{C_{l}\}_{1 \leq l \leq r}, \{\bar{C}_{l}\}_{1 \leq l \leq r}; B_{1},\bar{B}_{1},B_{2},\bar{B}_{2}, \{D_{l,1}\}_{1 \leq l \leq r},\{\bar{D}_{l,1}\}_{1 \leq l \leq r}, \{D_{l,2}\}_{1 \leq l \leq r},\{\bar{D}_{l,2}\}_{1 \leq l \leq r}
].
\]
For simplicity, we denote the system
\begin{equation}
\label{mflqsdg:homogeneous_sde}
\begin{cases}
dX(t) = \left\{ AX(t) + \bar{A}\mathbb{E}[X(t)] \right\} dt + \sum_{l=1}^{r}\left\{ C_{l}X(t) + \bar{C}_{l}\mathbb{E}[X(t)]  \right\} dw_l(t), \quad t\ge 0, \\
X(0) = x.
\end{cases}
\end{equation}
as 
\[
[A, \bar{A}, \{C_{l}\}_{1 \leq l \leq r}, \{\bar{C}_{l}\}_{1 \leq l \leq r}] = [A, \bar{A}, \{C_{l}\}_{1 \leq l \leq r}, \{\bar{C}_{l}\}_{1 \leq l \leq r}; 0, 0,0,0, \{0\}_{1 \leq l \leq r}, \{0\}_{1 \leq l \leq r},\{0\}_{1 \leq l \leq r}, \{0\}_{1 \leq l \leq r}],
\]
the system
$
[A, \{C_{l}\}_{1 \leq l \leq r}]= [A, 0,  \{C_{l}\}_{1 \leq l \leq r}, \{0\}_{1 \leq l \leq r}]
$
(the linear SDE without mean-field), and the system
$
[A ]= [A, 0, \{0\}_{1 \leq l \leq r}, \{0\}_{1 \leq l \leq r}] 
$
(linear ordinary differential equation, ODE for short).

\begin{definition}[\cite{Li2021}]
\label{mflqsdg:l2}
\begin{enumerate}
    \item[(i.)] System $[A, \bar{A}, \{C_{l}\}_{ 1 \leq l \leq r}, \{\bar{C}_{l}\}_{1 \leq l \leq r}]$ is said to be $L^2$-globally integrable, if for any $x\in\mathbb{R}^n$, the solution $X(\cdot)\equiv X(\cdot;x)$ of MF-SDE \eqref{mflqsdg:homogeneous_sde} is in $\mathcal{X}[0,\infty)$.
    \item[(ii.)] System $[A, \bar{A}, \{C_{l}\}_{ 1 \leq l \leq r}, \{\bar{C}_{l}\}_{1 \leq l \leq r}]$  is said to be $L^2$-asymptotically stable, if for any $x\in\mathbb{R}^n$, the solution $X(\cdot)\equiv X(\cdot;x)\in\mathcal{X}_{loc}[0,\infty)$ of MF-SDE \eqref{mflqsdg:homogeneous_sde} satisfies
    \[
    \lim_{t\to\infty} \mathbb{E}|X(t)|^2 = 0.
    \]
\end{enumerate}
\end{definition}

\begin{definition}[\cite{Li2021}]
\label{mflqsdg:l2+}
\begin{enumerate}
    \item[(i.)] For any $\Theta \equiv (\Theta_1,\Theta_2, \bar{\Theta}_1,\bar{\Theta}_2) \in \mathbb{R}^{m \times 4n}$,
    \begin{equation}
    \label{mflqsdg:l2_feedback}
    \begin{aligned}
    u_1(\cdot) \equiv \Theta_1 \{ X(\cdot) - \mathbb{E}[X(\cdot)] \} + \bar{\Theta}_1 \mathbb{E}[X(\cdot)], \\
    u_2(\cdot) \equiv \Theta_2 \{ X(\cdot) - \mathbb{E}[X(\cdot)] \} + \bar{\Theta}_2 \mathbb{E}[X(\cdot)],
     \end{aligned}
    \end{equation}
    is called a \textit{feedback control}. Under such a control, the state equation \eqref{mflqsdg:sde} becomes
    \begin{equation}
    \begin{cases}
    \begin{aligned}
    dX(t) &= \left\{ A^\Theta X(t) + \bar{A}^\Theta \mathbb{E}[X(t)] \right\} dt +  \sum_{l=1}^{r}\left\{ C_{l}^\Theta X(t) + \bar{C}_{l}^\Theta \mathbb{E}[X(t)]  \right\} dw_{l}(t), \quad t \ge 0, \\
    X(0) &= x,
    \end{aligned}
    \end{cases}
    \end{equation}
    where
    \begin{align*}
    A^\Theta &= A + B_1\Theta_1+B_2\Theta_2, & \bar{A}^\Theta &= \bar{A} + \bar{B}_1\bar{\Theta}_1 + B_1(\bar{\Theta}_1 - \Theta_1)+ \bar{B}_2\bar{\Theta}_2 + B_2(\bar{\Theta}_2 - \Theta_2), \\
    C_{l}^\Theta &= C_{l} + D_{l,1}\Theta_1+D_{l,2}\Theta_2, & \bar{C}_{l}^\Theta &= \bar{C}_{l} + \bar{D}_{l,1}\bar{\Theta}_1 + D_{l,1}(\bar{\Theta}_1 - \Theta_1)+\bar{D}_{l,2}\bar{\Theta}_2 + D_{l,2}(\bar{\Theta}_2 - \Theta_2).
    \end{align*}
    \item[(ii.)] System \[[A, \bar{A}, \{C_{l}\}_{1 \leq l \leq r}, \{\bar{C}_{l}\}_{1 \leq l \leq r}; B_{1},\bar{B}_{1},B_{2},\bar{B}_{2}, \{D_{l,1}\}_{1 \leq l \leq r},\{\bar{D}_{l,1}\}_{1 \leq l \leq r}, \{D_{l,2}\}_{1 \leq l \leq r},\{\bar{D}_{l,2}\}_{1 \leq l \leq r}
    ]\] is said to be \textit{MF-$L^2$-stabilizable}, if there exists a $\Theta \equiv(\Theta_1,\Theta_2, \bar{\Theta}_1,\bar{\Theta}_2) \in \mathbb{R}^{m \times 4n}$ such that system $[A^\Theta, \bar{A}^\Theta, C^\Theta, \bar{C}^\Theta]$ is $L^2$-asymptotically stable and system $[A^\Theta, C^\Theta]$ is $L^2$-globally integrable. In this case, $\Theta \equiv(\Theta_1,\Theta_2, \bar{\Theta}_1,\bar{\Theta}_2)$ is called an \textit{MF-$L^2$-stabilizer} of the system
    \[
    [A, \bar{A}, \{C_{l}\}_{1 \leq l \leq r}, \{\bar{C}_{l}\}_{1 \leq l \leq r}; B_{1},\bar{B}_{1},B_{2},\bar{B}_{2}, \{D_{l,1}\}_{1 \leq l \leq r},\{\bar{D}_{l,1}\}_{1 \leq l \leq r}, \{D_{l,2}\}_{1 \leq l \leq r},\{\bar{D}_{l,2}\}_{1 \leq l \leq r}
    ].
    \]
\end{enumerate}
\end{definition}

\subsection{Some useful results}

To facilitate subsequent analysis of the problem, we introduce the following decoupling representation method. Setting 
\[
\begin{bmatrix}  v_1(t) \\ v_2(t) \end{bmatrix}=\begin{bmatrix}  u_1(t) \\ u_2(t)  \end{bmatrix}-\begin{bmatrix} K_1(0) X(t)+(\hat{K}_1(0,0) -K_1(0))\mathbb{E}[X(t)])\\ K_2(0) X(t)+(\hat{K}_2(0,0) -K_2(0))\mathbb{E}[X(t)])\end{bmatrix},
\]
and $v_2(t) = LX(t)+(\hat{L}-L)\mathbb{E}[X(t)]$ in $({\text{MF-SDG}})^{\,0}_{\infty}$,
where
\[
\begin{bmatrix} K_1(0) \\ K_2(0) \end{bmatrix} = -R(0)^{-1} \begin{bmatrix} S_1(0) \\ S_2(0) \end{bmatrix},\begin{bmatrix} \hat{K}_1(0,0) \\ \hat{K}_2(0,0) \end{bmatrix} = -\hat{R}(0)^{-1} \begin{bmatrix} \hat{S}_1(0,0) \\ \hat{S}_2(0,0) \end{bmatrix}.
\]
Then, we obtain
\begin{equation*} 
    \begin{cases} 
        dX(t) = \left\{ A_{L}X(t) + \bar{A}_{L} \mathbb{E}[X(t)]+ B_1v_1(t) + \bar{B}_1\mathbb{E}[v_1(t)] \right\} dt \\
        \quad \quad+ \sum_{l=1}^{r}\left\{ C_{lL}X(t)+\bar{C}_{lL} \mathbb{E}[X(t)]+ D_{l,1}v_1(t) +\bar{D}_{l,1}\mathbb{E}[v_1(t)] \right\}  dw_{l}(t) \\
        X(0) = x
    \end{cases}
\end{equation*}
in which $x \in\mathbb{R}^n$, and 
\begin{equation*}
    J_{L}(x;v_1)=\mathbb{E}\int_{0}^{\infty}g(t,X(t),v_1(t),\mathbb{E}[X(t)],\mathbb{E}[v_1(t)])dt,
\end{equation*}
with
$g_{L}(t,x,v_1,\bar{x},\bar{v}_1)=\Bigg\langle\begin{pmatrix}
    Q_{L} & S^{\top}_{L} \\
    S_{L} & R_{11}
    \end{pmatrix}\begin{pmatrix}
    x\\
    v_1
    \end{pmatrix},\begin{pmatrix}
    x\\
    v_1
    \end{pmatrix}\Bigg\rangle+\Bigg\langle\begin{pmatrix}
    \bar{Q}_{L} & \bar{S}^{\top}_{L} \\
    \bar{S}_{L} & \bar{R}_{11}
    \end{pmatrix}\begin{pmatrix}
    \bar{x}\\
    \bar{v}_1
    \end{pmatrix},\begin{pmatrix}
    \bar{x}\\
    \bar{v}_1
    \end{pmatrix}\Bigg\rangle$
, where
\[
\begin{cases}
A_{L} = A + B_1K_1(0) + B_2K_2(0)+ B_2L\\
C_{lL} = C_{l} + D_{l,1}K_1(0) + D_{l,2}K_2(0)+ D_{l,2}L,  1 \leq l \leq r \\
Q_{L} = Q - S^{\top}(0) R(0)^{-1}S(0) + L^{\top}R_{22}L,\quad S_{L}= R_{12}L \\
\bar{A}_{L} = \bar{A} + \bar{B}_1\hat{K}_1(0,0) + \bar{B}_2\hat{K}_2(0,0)+B_1(\hat{K}_1(0,0)-K_1(0))+B_2(\hat{K}_2(0,0)-K_2(0))+ \bar{B}_2\hat{L}+B_2(\hat{L}-L)\\
\bar{C}_{lL} = \bar{C}_{l} + \bar{D}_{l,1}\bar{K}_1(0,0) + \bar{D}_{l,2}\bar{K}_2(0,0)+D_{l,1}(\hat{K}_1(0,0)-K_1(0))+D_{l,2}(\hat{K}_2(0,0)-K_2(0))+ \\
\quad \quad \quad\bar{D}_{l,2}\hat{L}+D_{l,2}(\hat{L}-L),  1 \leq l \leq r \\
\bar{Q}_{L} = \bar{Q} - \hat{S}^{\top}(0,0) \hat{R}(0)^{-1}\hat{S}(0,0) +S^{\top}(0) R(0)^{-1}S(0)+ \hat{L}^{\top}\hat{R}_{22}\hat{L}-L^{\top}R_{22}L,\quad \bar{S}_{L}= \hat{R}_{12}\hat{L} - R_{12}L 
\end{cases}.
\]
The corresponding coupled ARE is
\begin{equation}
\label{mflqsdg:coupledARE_l}
\begin{cases}
\begin{aligned}
&PA_{L} + A_{L}^{\top}P+ \sum_{l=1}^r C_{lL}^{\top}PC_{lL}+ Q_{L} - (B_{1}^{\top}P + \sum_{l=1}^r D_{l,1}^{\top}PC_{lL} + S_{L}) ^{\top} \\
& \quad \times(R_{11} + \sum_{l=1}^r D_{l,1}^{\top}PD_{l,1})^{-1} ( B_1^{\top}P + \sum_{l=1}^r D_{l,1}^{\top}PC_{lL} + S_{L} ) = 0,\\
&\hat{P}\hat{A}_{L} + \hat{A}_{L}^{\top}\hat{P}+  \sum_{l=1}^r \hat{C}_{lL}^{\top}P\hat{C}_{lL}+\hat{Q}_{L} - \big(\hat{B}_{1}^{\top}\hat{P} +\sum_{l=1}^r \hat{D}_{l,1}^{\top}P\hat{C}_{lL}+ \hat{S}_{L}\big) ^{\top} \\
& \quad \times(\hat{R}_{11}+\sum_{l=1}^r \hat{D}_{l,1}^{\top}P\hat{D}_{l,1})^{-1} \big( \hat{B}_{1}^{\top}\hat{P} +\sum_{l=1}^r \hat{D}_{l,1}^{\top}P\hat{C}_{lL}+\hat{S}_{L} \big) = 0.
\end{aligned} 
\end{cases}
\end{equation}

Throughout this work, $\mathcal{A}$ stands for the set of $(L,\hat{L})$, where $(L,\hat{L}) \in \mathbb{R}^{m_2 \times 2n} $ satisfy: 
\begin{itemize}
    \item The system $[A_{L}, \bar{A}_{L}, \{C_{lL}\}_{ 1 \leq l \leq r}, \{\bar{C}_{lL}\}_{1 \leq l \leq r};B_1,\bar{B}_1,\{D_{l,1}\}_{1 \leq l \leq r},\{\bar{D}_{l,1}\}_{1 \leq l \leq r}] $ is \textit{MF-$L^2$-stabilizable}.
    \item The coupled ARE \eqref{mflqsdg:coupledARE_l}  has a static stabilizing solution $(\tilde{P}_{L},\hat{P}_{L})$, satisfying the sign conditions
    \[
    R_{11} + \sum_{l=1}^r D_{l,1}^{\top}\tilde{P}_{L}D_{l,1} \prec 0,\quad \hat{R}_{11} + \sum_{l=1}^r \hat{D}_{l,1}^{\top}\tilde{P}_{L}\hat{D}_{l,1} \prec 0.
    \]
\end{itemize}

\begin{definition}[\cite{Dragan2013book}]
Given the following stochastic observation system:
\begin{equation*}
    \begin{cases}
         dX(t) = A X(t)dt + \sum_{l=1}^{r} C_l X(t)dw_l(t) \\
         dY(t) = E_0 X(t)dt + \sum_{l=1}^{r} E_l X(t)dw_l(t)
    \end{cases},
\end{equation*}
where $E_l \in R^{ q\times n}(0 \leq l \leq r)$, we denote this system by $\left[ {E_0},\{E_l\}_{ 1 \leq l \leq r}; A, \{C_l\}_{ 1 \leq l \leq r} \right]$. It is said to be \textit{stochastically detectable} if there exists a constant matrix $\varTheta \in \mathbb{R}^{n \times q}$ such that the system
$[A + \varTheta E_0,\ \{C_{l} + \varTheta E_{l}\}_{1 \leq l \leq r}]$
is $L^2$-asymptotically stable.
\end{definition}

\begin{lemma}[\cite{wang2025}]
\label{main_results:stochastically_detectable_lemma}
If the system $\left[ {E_0},\{E_l\}_{ 1 \leq l \leq r}; A, \{C_l\}_{ 1 \leq l \leq r} \right]$ is \textit{stochastically detectable}, then the following statements are equivalent:
\begin{itemize}
    \item[(a).] Let $\mathcal{L}^*$ denote the linear operator associated with this system, defined by
    $\mathcal{L}^*(P) = P A + A^{\top} P + \sum_{l=1}^{r} C_l^{\top} P C_l; \, \forall P \in \mathbb{S}^n$.
    Then all eigenvalues of $\mathcal{L}^*$ lie in the left half-plane, i.e., $\operatorname{Spec}\mathcal{L}^* \subset \mathbb{C}^{-}$.
    \item[(b).] There exists a matrix $P \in \overline{\mathbb{S}}_{+}^n$ satisfying the matrix equation $P A + A^{\top} P + \sum_{l=1}^{r} C_l^{\top} P C_l + \sum_{l=0}^{r} E_l^{\top} E_l = 0$.
\end{itemize}
\end{lemma}

\section{The Main Results}
\label{sec:main_results}

\subsection{\texorpdfstring{Analysis of Structural Characteristics for coupled AREs in $({\text{MF-SDG}})^{\,0}_{\infty}$}{Analysis of Structural Characteristics for MF-LQZSSDG}}
\label{subsec:function_gf}

From coupled AREs \eqref{mflqsdg:gtare}, we can define mapping $\mathcal{G}:\mathrm{Dom}\,\mathcal{G} \to \mathbb{S}^n$ and $\mathcal{F}:\mathrm{Dom}\,\mathcal{F} \to \mathbb{S}^n$,
\begin{equation}
\label{main_results:function_gf}
\begin{aligned}
    \mathcal{G}(P) &= P A + A^{\top} P + \sum_{l=1}^{r}C_{l}^{\top} P C_{l} + Q - S(P)^{\top} R(P)^{-1} S(P),\\
    \mathcal{F}(\hat{P},P) &= \hat{P}\hat{A} + \hat{A}^{\top}\hat{P}+\sum_{l=1}^{r}\hat{C}_{l}^{\top}P\hat{C}_{l} + \hat{Q}-\hat{S}(\hat{P},P)^{\top}\hat{R}(P)^{-1}\hat{S}(\hat{P},P). 
\end{aligned}
\end{equation}
The nonlinear function $\mathcal{G}$ and $\mathcal{F}$ are well-defined on the subsets,
\begin{align*}
    \mathrm{Dom}\,\mathcal{G} &=\left\{P\in\mathbb{S}^n| R_{11}(P) \prec 0, R_{22}(P) \succ 0\right\},\\
    \mathrm{Dom}\,\mathcal{F} &=\left\{(\hat{P},P)\in\mathbb{S}^n\times\mathbb{S}^n| R_{11}(P) \prec 0,\hat{R}_{11}(P) \prec 0,R_{22}(P) \succ 0,\hat{R}_{22}(P) \succ 0 \right\}.
\end{align*}
The aforementioned mapping $\mathcal{G}$ and  $\mathcal{F}$ serves as the fundamental foundation for constructing the nested iterative scheme. We begin by detailing its general properties, which collectively ensure the computational method's ability to operate iteratively.

First, we define following operators that will be used in the subsequent content:
\begin{equation}
\label{main_results:operators_k&n}
\begin{aligned}
\begin{bmatrix} K_1(P) \\ K_2(P) \end{bmatrix} &= -R(P)^{-1} \begin{bmatrix} S_1(P) \\ S_2(P) \end{bmatrix},
\begin{bmatrix} \hat{K}_1(\hat{P},P) \\ \hat{K}_2(\hat{P},P) \end{bmatrix} = -\hat{R}(P)^{-1} \begin{bmatrix} \hat{S}_1(\hat{P},P) \\ \hat{S}_2(\hat{P},P) \end{bmatrix},\\
\begin{bmatrix}
N_1(P,Z) \\
N_2(P,Z)
\end{bmatrix} &=
\begin{bmatrix}
B_1^{\top} Z + \sum_{l=1}^{r} D_{l,1}^{\top} Z \bigl(C_l + D_{l,1}K_1(P) + D_{l,2}K_2(P)\bigr) \\
B_2^{\top} Z + \sum_{l=1}^{r} D_{l,2}^{\top} Z \bigl(C_l + D_{l,1}K_1(P) + D_{l,2}K_2(P)\bigr)
\end{bmatrix},\\
\begin{bmatrix}
\hat{N}_1(\hat{P},P,\hat{Z},Z) \\
\hat{N}_2(\hat{P},P,\hat{Z},Z)
\end{bmatrix} &=
\begin{bmatrix}
\hat{B}_1^{\top} \hat{Z} + \sum_{l=1}^{r} \hat{D}_{l,1}^{\top} Z \bigl(\hat{C}_l + \hat{D}_{l,1}\hat{K}_1(\hat{P},P) + \hat{D}_{l,2}\hat{K}_2(\hat{P},P)\bigr) \\
\hat{B}_2^{\top} \hat{Z} + \sum_{l=1}^{r} \hat{D}_{l,2}^{\top} Z \bigl(\hat{C}_l + \hat{D}_{l,1}\hat{K}_1(\hat{P},P) + \hat{D}_{l,2}\hat{K}_2(\hat{P},P)\bigr)
\end{bmatrix},
\end{aligned}
\end{equation}
where $C_l,\hat{C}_l,B_i,D_{l,i},\hat{B}_i,\hat{D}_{l,i}(i = 1,2;1 \leq l \leq r)$ is defined in \eqref{mflqsdg:sde}, and $R(P), S_1(P),S_2(P),\hat{R}(P),\hat{S}_1(\hat{P},P),\\ \hat{S}_2(\hat{P},P)$ is defined in $\eqref{mflqsdg:gtare_auxiliarymatrix}$.

\begin{proposition}
\label{function_gf:proposition_1}
Let $P,Z,\hat{P}$ and $\hat{Z}$ be such that $P, P+Z \in \mathrm{Dom}\,\mathcal{G}$ and $(\hat{P},P),(\hat{P}+\hat{Z},P+Z)\in \mathrm{Dom}\,\mathcal{F}$. Then the feedback gains satisfy the following relation:
\begin{equation}
\label{feedback_1}
\begin{bmatrix} K_1(P+Z) \\ K_2(P+Z) \end{bmatrix} 
= \begin{bmatrix} K_1(P) \\ K_2(P) \end{bmatrix}-R(P+Z)^{-1}
\begin{bmatrix}
N_1(P,Z) \\
N_2(P,Z)
\end{bmatrix},
\end{equation}
\begin{equation}
\label{feedback_2}
\begin{bmatrix} \hat{K}_1(\hat{P}+\hat{Z},P+Z) \\ \hat{K}_2(\hat{P}+\hat{Z},P+Z) \end{bmatrix} 
= \begin{bmatrix} \hat{K}_1(\hat{P},P) \\ \hat{K}_2(\hat{P},P) \end{bmatrix}-\hat{R}(P+Z)^{-1}
\begin{bmatrix}
\hat{N}_1(\hat{P},P,\hat{Z},Z) \\
\hat{N}_2(\hat{P},P,\hat{Z},Z)
\end{bmatrix}.
\end{equation}
\end{proposition}

\begin{proof}
Equation \eqref{feedback_1} is derived from the conclusion of Proposition 3.1. in \cite{wang2025}. Then we verify the validity of Equation \eqref{feedback_2}.

Define the increment $\Delta = \begin{bmatrix} \hat{K}_1(\hat{P}+\hat{Z},P+Z) \\ \hat{K}_2(\hat{P}+\hat{Z},P+Z) \end{bmatrix} - \begin{bmatrix} \hat{K}_1(\hat{P},P) \\ \hat{K}_2(\hat{P},P) \end{bmatrix}$. 
By leveraging the definitions of $\hat{K}_1(\cdot,\cdot)$, $\hat{K}_2(\cdot,\cdot)$ in \eqref{main_results:operators_k&n}, and multiplying both sides of the equation by $\hat{R}(P+Z)$, we obtain:  
\[
\hat{R}(P+Z)\Delta = \hat{R}(P+Z) \hat{R}(P)^{-1} \begin{bmatrix} \hat{S}_1(\hat{P},P) \\ \hat{S}_1(\hat{P},P) \end{bmatrix} - \begin{bmatrix} \hat{S}_1(\hat{P}+\hat{Z},P+Z) \\ \hat{S}_1(\hat{P}+\hat{Z},P+Z) \end{bmatrix}.
\]
Note that,
\[
\hat{R}(P+Z) = \hat{R}(P) + \Delta \hat{R},\Delta \hat{R} = \begin{bmatrix}
\sum_{l=1}^{r} \hat{D}_{l,1}^{\top} Z \hat{D}_{l,1} & \sum_{l=1}^{r} \hat{D}_{l,1}^{\top} Z \hat{D}_{l,2} \\
\sum_{l=1}^{r} \hat{D}_{l,2}^{\top} Z \hat{D}_{l,1} & \sum_{l=1}^{r} \hat{D}_{l,2}^{\top} Z \hat{D}_{l,2}
\end{bmatrix}.
\]

Using the identity,
$
\hat{R}(P)^{-1} \begin{bmatrix} \hat{S}_1(\hat{P},P) \\ \hat{S}_2(\hat{P},P) \end{bmatrix} = -\begin{bmatrix} \hat{K}_1(\hat{P},P) \\ \hat{K}_2(\hat{P},P) \end{bmatrix},
$
we have
\[
\hat{R}(P+Z)\Delta = \begin{bmatrix} \hat{S}_1(\hat{P},P) \\ \hat{S}_2(\hat{P},P) \end{bmatrix}  - \Delta \hat{R} \begin{bmatrix} \hat{K}_1(\hat{P},P) \\ \hat{K}_2(\hat{P},P) \end{bmatrix}  - \begin{bmatrix} \hat{S}_1(\hat{P}+\hat{Z},P+Z) \\ \hat{S}_1(\hat{P}+\hat{Z},P+Z) \end{bmatrix}=-\begin{bmatrix}
\hat{N}_1(\hat{P},P,\hat{Z},Z) \\
\hat{N}_2(\hat{P},P,\hat{Z},Z)
\end{bmatrix} .
\]

Multiplying both sides by $\hat{R}(P+Z)^{-1}$ yields the relation.
\end{proof}

\begin{proposition}
\label{function_gf:proposition_2}
For any $P \in \mathrm{Dom}\,\mathcal{G}$,$(\hat{P},P) \in \mathrm{Dom}\,\mathcal{F}$ and for all $\varTheta_1,\hat{\varTheta}_1 \in \mathbb{R}^{m_1 \times n}, \varTheta_2,\hat{\varTheta}_2 \in \mathbb{R}^{m_2 \times n}$, we have
\begin{equation}
\label{function_gf:proposition_21}
\begin{aligned}
&\mathcal{G}(P) = P (A + B_1 \varTheta_1 + B_2 \varTheta_2) + (A + B_1 \varTheta_1 + B_2 \varTheta_2)^{\top} P + Q\\
& + \sum_{l=1}^{r} (C_{l} + D_{l,1} \varTheta_1 + D_{l,2} \varTheta_2)^{\top} P (C_{l} + D_{l,1} \varTheta_1 + D_{l,2} \varTheta_2)  + \begin{bmatrix} \varTheta_1 \\ \varTheta_2 \end{bmatrix}^{\top} R(0) \begin{bmatrix} \varTheta_1 \\ \varTheta_2\end{bmatrix} \\
&+ \begin{bmatrix} \varTheta_1 \\ \varTheta_2 \end{bmatrix}^{\top} \begin{bmatrix} S_1 \\ S_2 \end{bmatrix} + \begin{bmatrix} S_1 \\ S_2 \end{bmatrix}^{\top} \begin{bmatrix} \varTheta_1 \\ \varTheta_2 \end{bmatrix} - \begin{bmatrix} K_1(P) - \varTheta_1 \\ K_2(P) - \varTheta_2 \end{bmatrix}^{\top} R(P) \begin{bmatrix} K_1(P) - \varTheta_1 \\ K_2(P) - \varTheta_2\end{bmatrix},
\end{aligned}
\end{equation}
\begin{equation}
\label{function_gf:proposition_22}
\begin{aligned}
&\mathcal{F}(\hat{P},P) = \hat{P}(\hat{A}+\hat{B}_1 \hat{\varTheta}_1 + \hat{B}_2 \hat{\varTheta}_2) + (\hat{A}+\hat{B}_1 \hat{\varTheta}_1 + \hat{B}_2 \hat{\varTheta}_2)^{\top}\hat{P}+\hat{Q}\\
&+\sum_{l=1}^{r}(\hat{C}_{l} + \hat{D}_{l,1} \hat{\varTheta}_1 + \hat{D}_{l,2} \hat{\varTheta}_2)^{\top}P(\hat{C}_{l} + \hat{D}_{l,1} \hat{\varTheta}_1 + \hat{D}_{l,2} \hat{\varTheta}_2) + \begin{bmatrix} \hat{\varTheta}_1 \\ \hat{\varTheta}_2 \end{bmatrix}^{\top} \hat{R}(0) \begin{bmatrix} \hat{\varTheta}_1 \\ \hat{\varTheta}_2\end{bmatrix} \\
&+ \begin{bmatrix} \hat{\varTheta}_1 \\ \hat{\varTheta}_2 \end{bmatrix}^{\top} \begin{bmatrix} \hat{S}_1 \\ \hat{S}_2 \end{bmatrix} + \begin{bmatrix} \hat{S}_1 \\ \hat{S}_2 \end{bmatrix}^{\top} \begin{bmatrix} \hat{\varTheta}_1 \\ \hat{\varTheta}_2 \end{bmatrix} - \begin{bmatrix} \hat{K}_1(\hat{P},P) - \hat{\varTheta}_1 \\ \hat{K}_2(\hat{P},P)- \hat{\varTheta}_2 \end{bmatrix}^{\top} \hat{R}(P) \begin{bmatrix} \hat{K}_1(\hat{P},P) - \hat{\varTheta}_1 \\ \hat{K}_2(\hat{P},P)- \hat{\varTheta}_2 \end{bmatrix}.\\
\end{aligned}
\end{equation}
\end{proposition}

\begin{proof}
Equation \eqref{function_gf:proposition_21} is established in \cite{Dragan2013book}. We provide a rigorous proof for Equation \eqref{function_gf:proposition_22} as follows.
Let $\hat{\Theta} = \begin{bmatrix}\hat{\varTheta}_1 \\ \hat{\varTheta}_2\end{bmatrix}$, we start by expanding the key quadratic residual term:
\begin{align*} 
&(\hat{K}(\hat{P},P)-\hat{\Theta})^\top \hat{R}(P) (\hat{K}(\hat{P},P)-\hat{\Theta}) \\
=& \hat{K}(\hat{P},P)^\top\hat{R}(P)\hat{K}(\hat{P},P) -\hat{\Theta}^\top\hat{R}(P)\hat{K}(\hat{P},P) - \hat{K}(\hat{P},P)^\top\hat{R}(P)\hat{\Theta} + \hat{\Theta}^\top\hat{R}(P)\hat{\Theta}\\
=&\hat{Q}(\hat{P},P) + \hat{\Theta}^\top\hat{S}(\hat{P},P) + \hat{S}(\hat{P},P)^\top\hat{\Theta} + \hat{\Theta}^\top\hat{R}(P)\hat{\Theta},
\end{align*}
where $\hat{Q}(\hat{P},P),\hat{S}(\hat{P},P),\hat{R}(P)$ is defined in \eqref{mflqsdg:gtare_auxiliarymatrix}

Based on the definition $\hat{R}_{ij}(P) = \hat{R}_{ij}(0)+\sum_{l=1}^r\hat{D}_{l,i}^\top P\hat{D}_{l,j}$, decompose the quadratic term:
$$\hat{\Theta}^\top\hat{R}(P)\hat{\Theta} = \hat{\Theta}^\top\hat{R}(0)\hat{\Theta} + \sum_{l=1}^r \big(\hat{D}_{l,1}\hat{\varTheta}_1+\hat{D}_{l,2}\hat{\varTheta}_2\big)^\top P \big(\hat{D}_{l,1}\hat{\varTheta}_1+\hat{D}_{l,2}\hat{\varTheta}_2\big).$$
Substitute the definition of $\hat{S}(\hat{P},P)$ to expand the linear cross term:
\begin{align*} 
&\hat{\Theta}^\top\hat{S}(\hat{P},P)+\hat{S}(\hat{P},P)^\top\hat{\Theta} =\hat{P}(\hat{B}_1\hat{\varTheta}_1+\hat{B}_2\hat{\varTheta}_2)+(\hat{B}_1\hat{\varTheta}_1+\hat{B}_2\hat{\varTheta}_2)^\top\hat{P} + \hat{\Theta}^\top\begin{bmatrix}\hat{S}_1\\\hat{S}_2\end{bmatrix}+\begin{bmatrix}\hat{S}_1\\\hat{S}_2\end{bmatrix}^\top\hat{\Theta} \\ &+ \sum_{l=1}^r\left[\hat{C}_l^\top P(\hat{D}_{l,1}\hat{\varTheta}_1+\hat{D}_{l,2}\hat{\varTheta}_2)+(\hat{D}_{l,1}\hat{\varTheta}_1+\hat{D}_{l,2}\hat{\varTheta}_2)^\top P\hat{C}_l\right]. 
\end{align*}
Plug the above identities into the right-hand side of the expanded residual term, then rearrange all terms. Combining the deterministic state terms, stochastic summation terms, and exogenous quadratic/linear terms, we recover the full expression of $\mathcal{F}(\hat{P},P)$ in \eqref{function_gf:proposition_22}.
The proof is complete.
\end{proof}

\begin{proposition}
\label{function_gf:proposition_3}
Let $P,Z,\hat{P},\hat{Z}$ satisfy $P, P+Z \in \mathrm{Dom}\,\mathcal{G}$ and $(\hat{P},P),(\hat{P}+\hat{Z},P+Z)\in \mathrm{Dom}\,\mathcal{F}$.
Then the following identity holds:
\begin{equation}
\label{function_gf:proposition_31}
\begin{aligned}
&\mathcal{G}(P + Z) = \mathcal{G}(P)+Z(A + B_1K_1(P)+B_2K_2(P))+(A + B_1K_1(P)+B_2K_2(P))^{\top}Z\\
&+\sum_{l=1}^{r}(C_{l} + D_{l,1}K_1(P)+D_{l,2}K_2(P))^{\top}Z(C_{l} + D_{l,1}K_1(P)+D_{l,2}K_2(P)) \\
&- N(P,Z)^{\top} R(P+Z)^{-1} N(P,Z),\quad \text{where} \,\,N(P,Z)=\begin{bmatrix}
N_1(P,Z) \\
N_2(P,Z)
\end{bmatrix};
\end{aligned}
\end{equation}
\begin{equation}
\label{function_gf:proposition_32}
\begin{aligned}
&\mathcal{F}(\hat{P}+\hat{Z},P + Z) = \mathcal{F}(\hat{P},P) +\hat{Z}(\hat{A} + \hat{B}_1\hat{K}_1(\hat{P},P)+\hat{B}_2\hat{K}_2(\hat{P},P))+(\hat{A} + \hat{B}_1\hat{K}_1(\hat{P},P)+\hat{B}_2\hat{K}_2(\hat{P},P))^{\top}\hat{Z}\\
&+\sum_{l=1}^{r}(\hat{C}_{l} + \hat{D}_{l,1}\hat{K}_1(\hat{P},P)+\hat{D}_{l,2}\hat{K}_1(\hat{P},P))^{\top}Z(\hat{C}_{l} + \hat{D}_{l,1}\hat{K}_1(\hat{P},P)+\hat{D}_{l,2}\hat{K}_1(\hat{P},P)) \\
&- \hat{N}(\hat{P},P,\hat{Z},Z)^{\top} \hat{R}(P+Z)^{-1} \hat{N}(\hat{P},P,\hat{Z},Z),\quad \text{where} \,\,N(\hat{P},P,\hat{Z},Z)=\begin{bmatrix}
\hat{N}_1(\hat{P},P,\hat{Z},Z) \\
\hat{N}_2(\hat{P},P,\hat{Z},Z)
\end{bmatrix}.
\end{aligned}
\end{equation}
\end{proposition}

\begin{proof}
Equation \eqref{function_gf:proposition_31} is derived from the conclusion of Proposition 3.3. in \cite{wang2025}. Then we verify the validity of Equation \eqref{function_gf:proposition_32}.
Starting from the definition of $\mathcal{F}$,
$
\mathcal{F}(\hat{P}+\hat{Z},P+Z) = (\hat{P}+\hat{Z})\hat{A} + \hat{A}^{\top}(\hat{P}+\hat{Z}) + \sum_{l=1}^{r} \hat{C}_l^{\top} (P+Z) \hat{C}_l + \hat{Q} - \hat{S}(\hat{P}+\hat{Z},P+Z)^{\top} \hat{R}(P+Z)^{-1} \hat{S}(\hat{P}+\hat{Z},P+Z).
$

From Proposition \ref{function_gf:proposition_2}, expanding and rearranging terms, we obtain $\mathcal{F}(\hat{P}+\hat{Z},P+Z) =$
\[
\begin{aligned}
&\hat{P} \hat{A} + \hat{A}^{\top} \hat{P} + \sum_{l=1}^{r}\hat{C}_{l}^{\top} P \hat{C}_{l}+ \hat{Q} +\hat{K}(\hat{P},P)^{\top}\hat{R}(P)\hat{K}(\hat{P},P) + \hat{S}(\hat{P},P)^{\top}\hat{K}(\hat{P},P)+ \hat{K}(\hat{P},P)^{\top}\hat{S}(\hat{P},P)\\
&+\hat{Z}(\hat{A} + \hat{B}_1\hat{K}_1(\hat{P},P)+\hat{B}_2\hat{K}_2(\hat{P},P)) +(\hat{A} + \hat{B}_1\hat{K}_1(\hat{P},P)+\hat{B}_2\hat{K}_2(\hat{P},P))^{\top} \hat{Z} \\
&+ \sum_{l=1}^{r}(\hat{C}_{l} + \hat{D}_{l,1}\hat{K}_1(\hat{P},P)+\hat{D}_{l,2}\hat{K}_1(\hat{P},P))^{\top}Z(\hat{C}_{l} + \hat{D}_{l,1}\hat{K}_1(\hat{P},P)+\hat{D}_{l,2}\hat{K}_1(\hat{P},P))\\
&- (\hat{K}(\hat{P}+\hat{Z},P+Z)-\hat{K}(\hat{P},P))^{\top} \hat{R}(P+Z) (\hat{K}(\hat{P}+\hat{Z},P+Z)-\hat{K}(\hat{P},P)),\,\, \text{where} \,\, \hat{K}(\hat{P},P)=\begin{bmatrix}\hat{K}_1(\hat{P},P)\\\hat{K}_2(\hat{P},P)\end{bmatrix}.
\end{aligned}
\]

From Proposition \ref{function_gf:proposition_1}, we derive:
\[
\begin{aligned}
&(\hat{K}(\hat{P}+\hat{Z},P+Z)-\hat{K}(\hat{P},P))^{\top} \hat{R}(P+Z) (\hat{K}(\hat{P}+\hat{Z},P+Z)-\hat{K}(\hat{P},P))  \\
&= \hat{N}(\hat{P},P,\hat{Z},Z)^{\top} \hat{R}(P+Z)^{-1} \hat{N}(\hat{P},P,\hat{Z},Z).
\end{aligned}
\]
Substitute the above term into the rearranged expression of $\mathcal{F}(\hat{P}+\hat{Z},P+Z)$, and the target identity is thus verified to hold.
\end{proof}

Now we consider coupled ARE of the form
\begin{equation}
\label{pr:are1}
\begin{aligned}
PA + A^{\top}P + \sum_{l=1}^{r} C_{l}^{\top}PC_{l} + Q - ( B^{\top}P + \sum_{l=1}^{r} D_{l}^{\top}PC_{l}  )^{\top} ( R + \sum_{l=1}^{r} D_{l}^{\top}PD_{l} )^{-1} ( B^{\top}P + \sum_{l=1}^{r} D_{l}^{\top}PC_{l}  ) = 0,
\end{aligned}
\end{equation}
\begin{equation}
\label{pr:are2}
\begin{aligned}
\hat{P}\hat{A} + \hat{A}^{\top}\hat{P} + \sum_{l=1}^{r} \hat{C}_{l}^{\top}P\hat{C}_{l} + \hat{Q} - ( \hat{B}^{\top}\hat{P} + \sum_{l=1}^{r} \hat{D}_{l}^{\top}P\hat{C}_{l}  )^{\top} ( \hat{R} + \sum_{l=1}^{r} \hat{D}_{l}^{\top}P\hat{D}_{l} )^{-1} ( \hat{B}^{\top}\hat{P} + \sum_{l=1}^{r} \hat{D}_{l}^{\top}P\hat{C}_{l}  ) = 0,
\end{aligned}
\end{equation}
where $B, D_l,\hat{B},\hat{D}_{l} \in \mathbb{R}^{n \times m}$ for $1 \leq l \leq r$, and $R,\hat{R}\in \mathbb{R}^{m \times m}$. The remaining parameters follow the previous definitions.

\begin{proposition}
\label{main_results:coupled_ARE}
Suppose that the parameters of the coupled ARE \eqref{pr:are1} and \eqref{pr:are2} satisfy the following conditions:
\begin{itemize}
    \item $R\succ0$ and $\hat{R} \succ0$.
    \item The system 
    \[
    [A, \bar{A}, \{C_{l}\}_{1 \leq l \leq r}, \{\bar{C}_{l}\}_{1 \leq l \leq r}; B, \bar{B}, \{D_{l}\}_{1 \leq l \leq r}, \{\bar{D}_{l}\}_{1 \leq l \leq r}]
    \]
    is \textit{MF-$L^2$-stabilizable}.
    \item There exists a set of matrices $\{E_l\}_{0 \leq l \leq r}$ such that $\sum_{l=0}^{r} E_l^{\top}E_l = Q$, and a matrix $\hat{E}$ satisfying $\hat{E}^{\top}\hat{E} = \hat{Q}$, such that the systems
    \[
    \left[ E_0,\{E_l\}_{ 1 \leq l \leq r}; A, \{C_l\}_{ 1 \leq l \leq r} \right] \quad \text{and} \quad [\hat{E}; \hat{A}]
    \]
    are both \textit{detectable}.
\end{itemize}
Then, the coupled ARE\eqref{pr:are1} and \eqref{pr:are2} admit a unique pair of stabilizing solutions $(P,\hat{P})\in \overline{\mathbb{S}}_{+}^n \times \overline{\mathbb{S}}_{+}^n$ such that 
\[
R + \sum_{l=1}^{r} D_{l}^{\top}PD_{l} \succ 0 \quad \text{and} \quad \hat{R} + \sum_{l=1}^{r} \hat{D}_{l}^{\top}P\hat{D}_{l} \succ 0.
\]
\end{proposition}

\begin{proof}
According to Proposition 1 in \cite{Li2025}, the system 
\[
[A, \bar{A}, \{C_{l}\}_{1 \leq l \leq r}, \{\bar{C}_{l}\}_{1 \leq l \leq r}; B, \bar{B}, \{D_{l}\}_{1 \leq l \leq r}, \{\bar{D}_{l}\}_{1 \leq l \leq r}]
\]
is \textit{MF-$L^2$-stabilizable} if and only if the subsystems \[[A, \{C_{l}\}_{1 \leq l \leq r}; B, \{D_{l}\}_{1 \leq l \leq r} ]\] and $[\hat{A}; \hat{B}]$ are $L^2$-asymptotically stable.
By virtue of Proposition 1 in \cite{wang2025}, the considered ARE \eqref{pr:are1} admits a unique stabilizing solution $P\in \overline{\mathbb{S}}_{+}^n$ satisfying $R + \sum_{l=1}^{r} D_{l}^{\top}PD_{l} \succ 0$ and $\hat{R} + \sum_{l=1}^{r} \hat{D}_{l}^{\top}P\hat{D}_{l} \succ 0$.

Combining Proposition \ref{function_gf:proposition_1} and Proposition \ref{function_gf:proposition_2}, the coupled ARE \eqref{pr:are2} can be equivalently rewritten as
\[
\begin{aligned}
&\hat{P}\left( \hat{A}  + \hat{B} \hat{T}(0,P) \right) + \left( \hat{A}  + \hat{B} \hat{T}(0,P) \right)^{\top} \hat{P}+ \sum_{l=1}^{r} \left(\hat{C}_l  + \hat{D}_l \hat{T}(0,P)\right)^{\top} P \left(\hat{C}_l  + \hat{D}_l \hat{T}(0,P)\right)+\hat{Q} \\
&+\hat{T}(0,P)^{\top}\hat{R}(0)\hat{T}(0,P)-\big(\hat{T}(\hat{P},P)-\hat{T}(0,P)\big)^{\top}\hat{R}(P)\big(\hat{T}(\hat{P},P)-\hat{T}(0,P)\big) = 0,
\end{aligned}
\]
where the auxiliary matrix function is defined as
$
\hat{T}(\hat{P},P) = - \left( \hat{R} + \sum_{l=1}^{r} \hat{D}_{l}^{\top}P\hat{D}_{l} \right)^{-1} \left( \hat{B}^{\top}\hat{P} + \sum_{l=1}^{r} \hat{D}_{l}^{\top}P\hat{C}_{l} \right).
$
Further algebraic simplification yields
\[
\begin{aligned}
&\hat{P}\left( \hat{A}  + \hat{B} \hat{T}(0,P) \right) + \left( \hat{A}  + \hat{B} \hat{T}(0,P) \right)^{\top} \hat{P}+ \check{E}^{\top} \check{E} + \left( \hat{B}^{\top}\hat{P}\right)^{\top} \left( \hat{R} + \sum_{l=1}^{r} \hat{D}_{l}^{\top}P\hat{D}_{l} \right)^{-1} \left( \hat{B}^{\top}\hat{P} \right) = 0,
\end{aligned}
\]
where the matrix $\check{E}$ is given by
\[
\check{E} = 
\begin{bmatrix}
\hat{E} \\
\sqrt{\hat{R}(0)}\hat{T}(0,P)\\
\sqrt{\sum_{l=1}^{r} \left(\hat{C}_l  + \hat{D}_l \hat{T}(0,P)\right)^{\top} P \left(\hat{C}_l  + \hat{D}_l \hat{T}(0,P)\right)} 
\end{bmatrix} \in \mathbb{R}^{(q+m+n) \times n}.
\]
Since the subsystem $[\hat{E}; \hat{A}]$ is detectable by assumption, it follows immediately that the augmented system $[\check{E}; \hat{A}]$ is also detectable. Applying Lemma \ref{main_results:stochastically_detectable_lemma}, we conclude that the coupled ARE possesses a unique stabilizing solution $\hat{P}\in \overline{\mathbb{S}}_{+}^n$. This completes the proof.
\end{proof}

\subsection{Convergence Analysis of Iterative Sequence}

For each triple $(P,\hat{P},Z,\hat{Z},L,\hat{L})$ satisfying $P,P + Z\in\mathrm{Dom}\,\mathcal{G}$, $(\hat{P},P),(\hat{P}+\hat{Z},P+Z)\in\mathrm{Dom}\,\mathcal{F}$, and $(L,\hat{L}) \in \mathcal{A}$, the linear operators $\mathcal{L}^{*}_{P}$, $\mathcal{L}^{*}_{P+Z}$, $\hat{\mathcal{L}}^{*}_{(\hat{P},P)}$, and $\hat{\mathcal{L}}^{*}_{(\hat{P}+\hat{Z},P+Z)}$ are defined as follows:
\begin{equation}
\label{linear_operators_pl}
\begin{aligned}
    &\mathcal{L}^{*}_{P}(Y) = Y ( A_{(0)}+ B_1\check{K}_1(P)+B_2L)+( A_{(0)}+ B_1\check{K}_1(P)+B_2L )^{\top}Y +\\
    &\sum_{l=1}^r (C_{l,(0)}+D_{l,1}\check{K}_1(P)+D_{l,2}L)^{\top} Y (C_{l,(0)}+D_{l,1}\check{K}_1(P)+D_{l,2}L), \, \forall \,Y \in \mathbb{S}^n,
\end{aligned}
\end{equation}
\begin{equation}
\label{linear_operators_pzl}
\begin{aligned}
    &\mathcal{L}^{*}_{P+Z}(Y) = Y ( A_{(0)}+ B_1\check{K}_1(P+Z)+B_2L)+( A_{(0)}+ B_1\check{K}_1(P+Z)+B_2L)^{\top} Y +\\
    & \sum_{l=1}^r (C_{l,(0)}+ D_{l,1}\check{K}_1(P+Z)+D_{l,2}L)^{\top} Y (C_{l,(0)}+D_{l,1}\check{K}_1(P+Z)+D_{l,2}L), \, \forall \, Y \in \mathbb{S}^n,
\end{aligned}
\end{equation}
\begin{equation}
\label{linear_operators_hpl}
\begin{aligned}
    &\hat{\mathcal{L}}^{*}_{(\hat{P},P)}(Y) = Y ( \hat{A}_{(0)}+ \hat{B}_1\check{K}_1(\hat{P},P)+\hat{B}_2\hat{L})+( \hat{A}_{(0)}+ \hat{B}_1\check{K}_1(\hat{P},P)+\hat{B}_2\hat{L} )^{\top}Y , \, \forall \,Y \in \mathbb{S}^n,
\end{aligned}
\end{equation}
\begin{equation}
\label{linear_operators_hpzl}
\begin{aligned}
    &\hat{\mathcal{L}}^{*}_{(\hat{P}+\hat{Z},P+Z)}(Y) = Y ( \hat{A}_{(0)}+ \hat{B}_1\check{K}_1(\hat{P}+\hat{Z},P+Z)+\hat{B}_2\hat{L})+\\
    &\quad \quad ( \hat{A}_{(0)}+ \hat{B}_1\check{K}_1(\hat{P}+\hat{Z},P+Z)+\hat{B}_2\hat{L})^{\top} Y , \, \forall \, Y \in \mathbb{S}^n,
\end{aligned}
\end{equation}
where
\begin{equation}
\label{iterative:matrices_evolve0}
\begin{cases}
        A_{(0)} = A +  \begin{bmatrix} B_1 & B_2 \end{bmatrix} \begin{bmatrix} K_1(0) \\ K_2(0) \end{bmatrix} ,\hat{A}_{(0)} = \hat{A} +  \begin{bmatrix} \hat{B}_1 & \hat{B}_2 \end{bmatrix} \begin{bmatrix} \hat{K}_1(0,0) \\ \hat{K}_2(0,0) \end{bmatrix},\\
        C_{l,(0)} = C_{l} + \begin{bmatrix} D_{l1} & D_{l2} \end{bmatrix} \begin{bmatrix} K_1(0) \\ K_2(0) \end{bmatrix}, \hat{C}_{l,(0)} = \hat{C}_{l} + \begin{bmatrix} \hat{D}_{l1} & \hat{D}_{l2} \end{bmatrix} \begin{bmatrix} \hat{K}_1(0,0) \\ \hat{K}_2(0,0) \end{bmatrix},  \quad  1 \leq l \leq r,\\
        \begin{bmatrix} \check{K}_1(P) \\ \check{K}_2(P) \end{bmatrix}
= \begin{bmatrix} K_1(P) \\ K_2(P) \end{bmatrix}-\begin{bmatrix} K_1(0) \\ K_2(0) \end{bmatrix},
\begin{bmatrix} \check{K}_1(\hat{P},P) \\ \check{K}_2(\hat{P},P) \end{bmatrix}
= \begin{bmatrix} \hat{K}_1(\hat{P},P) \\ \hat{K}_2(\hat{P},P) \end{bmatrix}-\begin{bmatrix} \hat{K}_1(0,0) \\ \hat{K}_2(0,0) \end{bmatrix}.
\end{cases}
\end{equation}
and $\begin{bmatrix} K_1(P) \\ K_2(P) \end{bmatrix},\begin{bmatrix} \hat{K}_1(\hat{P},P) \\ \hat{K}_2(\hat{P},P) \end{bmatrix}$ are defined in \eqref{main_results:operators_k&n}.

\begin{lemma}
\label{lemma:coupled_joint_compare}
Assume $0\in \mathrm{Dom}\, \mathcal{G}$, $(0,0)\in \mathrm{Dom}\, \mathcal{F}$ and $(L,\hat{L})\in\mathcal{A}$. Let $P,Z,\hat{P},\hat{Z}\in \mathbb{S}^n$ satisfy:
\begin{enumerate}
    \item $P,P+Z\in\mathrm{Dom}\,\mathcal{G}$,\quad $(\hat{P},P),(\hat{P}+\hat{Z},P+Z)\in\mathrm{Dom}\,\mathcal{F}$;
    \item $P,Z$ satisfy 
    \begin{equation}
    \label{main_results:g_interrelated_equation}
    \begin{aligned}
        & 0=\mathcal{G}(P)+Z(A_{(0)} + B_1\check{K}_1(P) + B_2\check{K}_2(P) )+(A_{(0)} +  B_1\check{K}_1(P) + B_2\check{K}_2(P))^{\top}Z \\
        &\quad +\sum_{l=1}^{r}(C_{l,(0)} +  D_{l,1}\check{K}_1(P) + D_{l,2}\check{K}_2(P))^{\top}Z(C_{l,(0)} +  D_{l,1}\check{K}_1(P) + D_{l,2}\check{K}_2(P))\\
        &\quad -\left[B_2^{\top} Z + \sum_{l=1}^{r}D_{l,2}^{\top} Z (C_{l,(0)} +  D_{l,1}\check{K}_1(P) + D_{l,2}\check{K}_2(P))\right]^{\top} R_{22}(P+Z)^{-1}\\
        &\quad \times \left[B_2^{\top} Z + \sum_{l=1}^{r}D_{l,2}^{\top} Z (C_{l,(0)} +  D_{l,1}\check{K}_1(P) + D_{l,2}\check{K}_2(P))\right];
    \end{aligned}
    \end{equation}
    \item The coupled pair $(P,\hat{P})$ and the correction pair $(Z,\hat{Z})$ satisfy
    \begin{equation}
    \label{main_results:f_interrelated_equation}
    \begin{aligned}
        0=&\mathcal{F}(\hat{P},P)+\hat{Z}\left(\hat{A}_{(0)} + \hat{B}_1\check{K}_1(\hat{P},P) + \hat{B}_2\check{K}_2(\hat{P},P) \right) +\left(\hat{A}_{(0)} +  \hat{B}_1\check{K}_1(\hat{P},P) + \hat{B}_2\check{K}_2(\hat{P},P)\right)^{\top}\hat{Z}\\
        &+\sum_{l=1}^{r}\left(\hat{C}_{l,(0)} +  \hat{D}_{l,1}\check{K}_1(\hat{P},P) + \hat{D}_{l,2}\check{K}_2(\hat{P},P)\right)^{\top}Z\left(\hat{C}_{l,(0)} +  \hat{D}_{l,1}\check{K}_1(\hat{P},P) + \hat{D}_{l,2}\check{K}_2(\hat{P},P)\right)\\
        &-\left[\hat{B}_2^{\top} \hat{Z} + \sum_{l=1}^{r}\hat{D}_{l,2}^{\top} Z \left(\hat{C}_{l,(0)} +  \hat{D}_{l,1}\check{K}_1(\hat{P},P) + \hat{D}_{l,2}\check{K}_2(\hat{P},P)\right)\right]^{\top} \hat{R}_{22}(P+Z)^{-1}\\
        & \quad \times \left[\hat{B}_2^{\top} \hat{Z} + \sum_{l=1}^{r}\hat{D}_{l,2}^{\top} Z \left(\hat{C}_{l,(0)} +  \hat{D}_{l,1}\check{K}_1(\hat{P},P) + \hat{D}_{l,2}\check{K}_2(\hat{P},P)\right)\right].
    \end{aligned}
    \end{equation}
\end{enumerate}
Denote $(P_{L},\hat{P}_{L})$ as the stabilizing solution of the coupled ARE \eqref{mflqsdg:coupledARE_l}. Then the following joint bidirectional implications hold:
\begin{enumerate}
    \item[(\romannumeral1).] If \,$\operatorname{Spec} \,\mathcal{L}^*_{P}\subset\mathbb{C}^{-}$ and $\operatorname{Spec}\,\hat{\mathcal{L}}^{*}_{(\hat{P},P)}\subset\mathbb{C}^{-}$, then
    $$\tilde{P}_L\succeq P+Z,\quad \hat{P}_{L}\succeq \hat{P}+\hat{Z};$$
    \item[(\romannumeral2).] If $\tilde{P}_L\succeq P+Z$ and $\hat{P}_{L}\succeq \hat{P}+\hat{Z}$, then
    $$\operatorname{Spec} \,\mathcal{L}^*_{P+Z}\subset\mathbb{C}^{-},\quad \operatorname{Spec}\,\hat{\mathcal{L}}^{*}_{(\hat{P}+\hat{Z},P+Z)}\subset\mathbb{C}^{-}.$$
\end{enumerate}
\end{lemma}

\begin{proof}
\textbf{Proof of (\romannumeral1).}
\textbf{Part 1: Single-system part.} 
The hypotheses $0\in \mathrm{Dom}\, \mathcal{G}$, $(L,\hat{L})\in\mathcal{A}$, $P,P+Z\in\mathrm{Dom}\, \mathcal{G}$ together with \eqref{main_results:g_interrelated_equation} and $\operatorname{Spec} \,\mathcal{L}^*_{P}\subset\mathbb{C}^{-}$ fulfill exactly the assumptions of Lemma~3.8~(i) in \cite{wang2025}. Hence we directly obtain
\[
\tilde{P}_L \succeq P+Z .
\]

\noindent \textbf{Part 2: Coupled system part.} 
By the expansion properties of $\mathcal{F}$, Propositions~\ref{function_gf:proposition_2} and~\ref{function_gf:proposition_3} for $\mathcal{F}$, we have
\begin{equation}
\label{coupled_expansion}
\begin{aligned}
&\mathcal{F}(\hat{P}+\hat{Z},P+Z)=\hat{\mathcal{L}}^{*}_{(\hat{P},P)}(\hat{P}+\hat{Z}) + \Xi_{L} -\hat{\Phi}^{\top}\hat{R}(P+Z)\hat{\Phi},
\end{aligned}
\end{equation}
where $\hat{\Phi}=J(\hat{P},P)-\check{K}(\hat{P}+\hat{Z},P+Z)$ and 
\[
\begin{aligned}
\Xi_{L} & = \hat{Q}_{L}+\check{K}_1^{\top}(\hat{P},P)\hat{R}_{12}\hat{L}+\hat{L}^{\top}\hat{R}_{21}\check{K}_1(\hat{P},P)+\check{K}_1^{\top}(\hat{P},P)\hat{R}_{11}\check{K}_1(\hat{P},P)\\
&\quad +\sum_{l=1}^{r}\bigl(\hat{C}_{l,(0)} +  \hat{D}_{l,1}\check{K}_1(\hat{P},P) + \hat{D}_{l,2}\hat{L}\bigr)^{\top}Z\bigl(\hat{C}_{l,(0)} +  \hat{D}_{l,1}\check{K}_1(\hat{P},P) + \hat{D}_{l,2}\hat{L}\bigr),
\end{aligned}
\]
$J(\hat{P},P)=\begin{bmatrix}\check{K}_1(\hat{P},P) \\ \hat{L}\end{bmatrix}$, $\hat{\mathcal{L}}^{*}_{(\hat{P},P)}$ is deined in \eqref{linear_operators_hpl}, $\check{K}_1(\cdot,\cdot)$ and $\check{K}(\cdot,\cdot)$ is the vector of gain matrices defined in \eqref{iterative:matrices_evolve0}.

From Proposition~\ref{function_gf:proposition_3} and the coupled equations \eqref{main_results:f_interrelated_equation}, \eqref{coupled_expansion}, the pair $(\hat{P}+\hat{Z},P+Z)$ is shown to satisfy a Riccati-type equation:
\begin{equation}
\label{coupled_ric_eq}
\hat{\mathcal{L}}^{*}_{(\hat{P},P)}(\hat{P}+\hat{Z}) + \Xi_{L} - \hat{\Phi}^{\top}\hat{R}(P+Z)\hat{\Phi} + \hat{\Lambda}^{\top}\hat{\mathbb{R}}_{22}^{\sharp}(P+Z)^{-1}\hat{\Lambda}=0,
\end{equation}
where
\[
\hat{\mathbb{R}}_{22}^{\sharp}(P+Z)=\hat{R}_{11}(P+Z)-\hat{R}_{12}(P+Z)\hat{R}_{22}(P+Z)^{-1}\hat{R}_{21}(P+Z),\quad \hat{\Lambda}=\hat{N}_1-\hat{R}_{12}(P+Z)\hat{R}_{22}(P+Z)^{-1}\hat{N}_2,
\]
\[
\begin{aligned}
\hat{N}_1&=\hat{B}_1^{\top} \hat{Z} + \sum_{l=1}^{r}\hat{D}_{l,1}^{\top} Z \bigl(\hat{C}_{l,(0)} + \hat{D}_{l,1}\check{K}_1(\hat{P},P) +\hat{D}_{l,2}\check{K}_2(\hat{P},P) \bigr),\\
\hat{N}_2&=\hat{B}_2^{\top} \hat{Z} + \sum_{l=1}^{r}\hat{D}_{l,2}^{\top} Z \bigl(\hat{C}_{l,(0)} + \hat{D}_{l,1}\check{K}_1(\hat{P},P) +\hat{D}_{l,2}\check{K}_2(\hat{P},P) \bigr).
\end{aligned}
\]

On the other hand, the stabilizing solution $(P_{L},\hat{P}_{L})$ of the coupled ARE system \eqref{mflqsdg:coupledARE_l} obeys the equation (again by the analogue of Proposition~\ref{function_gf:proposition_2})
\begin{equation}
\label{coupled_sol_eq}
\hat{\mathcal{L}}^{*}_{(\hat{P},P)}(\hat{P}_{L}) + \Xi_{L} - \hat{\Psi}^{\top}\hat{\Omega}\hat{\Psi}=0,
\end{equation}
where
\[
\hat{\Psi}=\check{K}_1(\hat{P},P) - \check{K}_{L}(\hat{P}_{L},\tilde{P}_L),\qquad
\hat{\Omega}= \hat{R}_{11} + \sum_{l=1}^{r} \hat{D}_{l,1}^{\top}\tilde{P}_{L}\hat{D}_{l,1},
\]
and $\check{K}_{L}(\hat{P}_{L},\tilde{P}_L)=
-\left(\hat{R}_{11} + \sum_{l=1}^r \hat{D}_{l,1}^{\top}\tilde{P}_{L}\hat{D}_{l,1}\right)^{-1}\left(\hat{B}^{\top}_1\hat{P}_{L} + \sum_{l=1}^r D^{\top}_{l,1}\tilde{P}_{L}C_{l} + \hat{S}_{L}\right).
$

Now we decompose the quadratic term appearing in \eqref{coupled_ric_eq}. Exactly as in the single-system proof (Lemma~3.8) in \cite{wang2025}, one verifies that
\begin{equation}
\label{coupled_decomp}
\hat{\Phi}^{\top}\hat{R}(P+Z)\hat{\Phi}
= \hat{H}_1^{\top}\hat{H}_1 + \hat{\Lambda}^{\top}\hat{\mathbb{R}}_{22}^{\sharp}(P+Z)^{-1}\hat{\Lambda},\quad \hat{H}_1 = \begin{bmatrix}
\hat{R}_{22}(P+Z)^{-\frac12}\hat{R}_{21}(P+Z) & \hat{R}_{22}(P+Z)^{\frac12}
\end{bmatrix} \hat{\Phi}.
\end{equation}

Insert \eqref{coupled_decomp} into \eqref{coupled_ric_eq}; the $\hat{\Lambda}$-terms cancel and we obtain
\begin{equation}
\label{coupled_ric_simple}
\hat{\mathcal{L}}^{*}_{(\hat{P},P)}(\hat{P}+\hat{Z}) + \Xi_{L} - \hat{H}_1^{\top}\hat{H}_1 = 0.
\end{equation}
Subtracting \eqref{coupled_ric_simple} from \eqref{coupled_sol_eq} yields
\[
\hat{\mathcal{L}}^{*}_{(\hat{P},P)}(\Delta_{\hat{P}}) + \bigl(\hat{H}_1^{\top}\hat{H}_1 - \hat{\Psi}^{\top}\hat{\Omega}\hat{\Psi}\bigr) = 0,
\qquad \Delta_{\hat{P}} := \hat{P}_{L} - (\hat{P}+\hat{Z}).
\]
Since by hypothesis $\operatorname{Spec}\,\hat{\mathcal{L}}^{*}_{(\hat{P},P)}\subset\mathbb{C}^{-}$ and $\hat{H}_1^{\top}\hat{H}_1 - \hat{\Psi}^{\top}\hat{\Omega}\hat{\Psi} \succeq 0$, the Theorem 3.2.2. and Theorem 2.7.5 in \cite{Dragan2013book} immediately gives $\Delta_{\hat{P}}\succeq 0$, i.e.
\[
\hat{P}_{L} \succeq \hat{P}+\hat{Z}.
\]
Thus the proof of (\romannumeral1) is complete.
 
\textbf{Proof of (\romannumeral2).}
\textbf{Part 1: Single-system part.} 
The assumptions $\tilde{P}_L\succeq P+Z$ together with the interrelated equation \eqref{main_results:g_interrelated_equation} and the domain conditions are exactly those of Lemma~3.8~(ii) in \cite{wang2025}. Hence
\[
\operatorname{Spec} \,\mathcal{L}^*_{P+Z}\subset\mathbb{C}^{-}.
\]
 
\textbf{Part 2: Coupled system part.}
We now rewrite the coupled ARE at the shifted point $(\hat{P}+\hat{Z},P+Z)$. By the same expansion property,
\begin{equation}
\label{coupled_ric_eq_shift}
\begin{aligned}
\hat{\mathcal{L}}^{*}_{(\hat{P}+\hat{Z},P+Z)}(\hat{P}+\hat{Z}) + \Xi_{L,+}
&- \hat{\Phi}_+^{\top}\hat{R}(P+Z)\hat{\Phi}_+ + \hat{\Lambda}^{\top}\hat{\mathbb{R}}_{22}^{\sharp}(P+Z)^{-1}\hat{\Lambda}=0,
\end{aligned}
\end{equation}
where $\Xi_{L,+}$ is obtained from $\Xi_{L}$ by replacing $(\hat{P},P)$ with $(\hat{P}+\hat{Z},P+Z)$, and
\[
\hat{\Phi}_+ = J(\hat{P}+\hat{Z},P+Z)-\check{K}(\hat{P}+\hat{Z},P+Z)
= \begin{bmatrix} 0 \\ \hat{L} - \check{K}_2(\hat{P}+\hat{Z},P+Z) \end{bmatrix}.
\]
Note that $\hat{\Lambda}$ is unchanged because it is built from $\hat{N}_1,\hat{N}_2$ which are fixed by the original $(\hat{P},P)$ and $(\hat{Z},Z)$.  
Applying the decomposition \eqref{coupled_decomp} at $(\hat{P}+\hat{Z},P+Z)$ and observing that the first block of $\hat{\Phi}_+$ vanishes, we obtain simply
\begin{equation}
\label{coupled_decomp_shift}
\hat{\Phi}_+^{\top}\hat{R}(P+Z)\hat{\Phi}_+ = \hat{H}_2^{\top}\hat{H}_2,\quad \hat{H}_2 = \begin{bmatrix}
\hat{R}_{22}(P+Z)^{-\frac12}\hat{R}_{21}(P+Z) & \hat{R}_{22}(P+Z)^{\frac12}
\end{bmatrix} \hat{\Phi}_+ .
\end{equation}

Substituting \eqref{coupled_decomp_shift} into \eqref{coupled_ric_eq_shift} gives
\begin{equation}
\label{coupled_ric_shift_simple}
\hat{\mathcal{L}}^{*}_{(\hat{P}+\hat{Z},P+Z)}(\hat{P}+\hat{Z}) + \Xi_{L,+} - \hat{H}_2^{\top}\hat{H}_2 + \hat{\Lambda}^{\top}\hat{\mathbb{R}}_{22}^{\sharp}(P+Z)^{-1}\hat{\Lambda}=0 .
\end{equation}
The stabilizing solution $\hat{P}_{L}$ satisfies, at the same shifted point,
\begin{equation}
\label{coupled_sol_eq_shift}
\hat{\mathcal{L}}^{*}_{(\hat{P}+\hat{Z},P+Z)}(\hat{P}_{L}) + \Xi_{L,+} - \hat{\Psi}_+^{\top}\hat{\Omega}\hat{\Psi}_+=0,
\end{equation}
where
\[
\hat{\Psi}_+ = \check{K}_1(\hat{P}+\hat{Z},P+Z) - \check{K}_{L}(\hat{P}_{L},\tilde{P}_L),
\]
and $\hat{\Omega}$ is the same matrix as in Part~1 (it depends only on the stabilizing solution $P_{L}$).

Subtract \eqref{coupled_ric_shift_simple} from \eqref{coupled_sol_eq_shift} and recall $\Delta_{\hat{P}} = \hat{P}_{\hat{L}} - (\hat{P}+\hat{Z})$, we arrive at
\[
\hat{\mathcal{L}}^{*}_{(\hat{P}+\hat{Z},P+Z)}(\Delta_{\hat{P}}) + \hat{H}_2^{\top}\hat{H}_2 - \hat{\Psi}_+^{\top}\hat{\Omega}\hat{\Psi}_+ - \hat{\Lambda}^{\top}\hat{\mathbb{R}}_{22}^{\sharp}(P+Z)^{-1}\hat{\Lambda}=0 .
\]
Now construct the matrices $\hat{E}$ as:
\[
\hat{E} = \begin{bmatrix}
\hat{H}_2 \\
\sqrt{-\,\hat{\Omega}}\; \hat{\Psi}_+ \\
\sqrt{-\hat{\mathbb{R}}_{22}^{\sharp}(P+Z)^{-1}}\; \hat{\Lambda}
\end{bmatrix}.
\]
A direct calculation yields
$
\hat{E}^{\top}\hat{E}
= \hat{H}_2^{\top}\hat{H}_2 - \hat{\Psi}_+^{\top}\hat{\Omega}\hat{\Psi}_+ - \hat{\Lambda}^{\top}\hat{\mathbb{R}}_{22}^{\sharp}(P+Z)^{-1}\hat{\Lambda},
$
and consequently
\[
\hat{\mathcal{L}}^{*}_{(\hat{P}+\hat{Z},P+Z)}(\Delta_{\hat{P}}) + \hat{E}^{\top}\hat{E} = 0.
\]
Define the auxiliary matrix
\[
\hat{\varTheta} = \begin{bmatrix}
\mathbb{O}_{n \times m_2}^{\top} \\
-\sqrt{-\hat{\Omega}^{-1}}^{\top}\hat{B}_1^{\top} \\
\mathbb{O}_{n \times m_1}^{\top}
\end{bmatrix}^{\top}.
\]
Then the augmented system
\[
\bigl[\hat{A}_{(0)} + \hat{B}_1\check{K}_1(\hat{P}+\hat{Z},P+Z) + \hat{B}_2\hat{L} + \hat{\varTheta}\hat{E}\bigr]
\]
is stable, which implies that the system $[\hat{E}; \hat{A}_{(0)} + \hat{B}_1\check{K}_1(\hat{P}+\hat{Z},P+Z) + \hat{B}_2\hat{L}]$ is detectable. By the Lemma~\ref{main_results:stochastically_detectable_lemma} we conclude
\[
\operatorname{Spec}\,\hat{\mathcal{L}}^{*}_{(\hat{P}+\hat{Z},P+Z)}\subset\mathbb{C}^{-}.
\]
This completes the proof of (\romannumeral2).
\end{proof}

Before presenting the coupled counterpart of Theorem 3.9 in \cite{wang2025}, we introduce the linear operators associated with the coupled iteration sequences.  
For each $k\ge 0$, let $P^{(k)},\hat P^{(k)},Z^{(k)},\hat Z^{(k)}\in\mathbb{S}^n$. Define
\begin{equation}
\label{linear_operators_pkl}  
\begin{aligned}
&\mathcal{L}^*_{P^{(k)}}(Y) =
Y\bigl(A_{(0)}+B_1\check K_1(P^{(k)})+B_2L\bigr)
+\bigl(A_{(0)}+B_1\check K_1(P^{(k)})+B_2L\bigr)^\top Y \notag\\
&\quad +\sum_{l=1}^r \bigl(C_{l,(0)}+D_{l,1}\check K_1(P^{(k)})+D_{l,2}L\bigr)^\top Y
\bigl(C_{l,(0)}+D_{l,1}\check K_1(P^{(k)})+D_{l,2}L\bigr), \, \forall \,Y \in \mathbb{S}^n,
\end{aligned} 
\end{equation}
\begin{equation}
\label{linear_operators_k&+1}
\begin{aligned}
&\mathcal{L}^{(k,k+1)}(Y) =
Y\bigl(A_{(k)}+B_2T_{(k+1)}\bigr)+\bigl(A_{(k)}+B_2T_{(k+1)}\bigr)^\top Y \notag\\
&\quad +\sum_{l=1}^r \bigl(C_{l,(k)}+D_{l,2}T_{(k+1)}\bigr)^\top Y
\bigl(C_{l,(k)}+D_{l,2}T_{(k+1)}\bigr),  \, \forall \,Y \in \mathbb{S}^n,
\end{aligned}
\end{equation}
and for the coupled system
\begin{equation}
\label{linear_operators_hpkl}  
\begin{aligned}
\hat{\mathcal{L}}^*_{(\hat P^{(k)},P^{(k)})}(Y) =
Y\bigl(\hat A_{(0)}+\hat B_1\check K_1(\hat P^{(k)},P^{(k)})+\hat B_2\hat L\bigr)+\bigl(\hat A_{(0)}+\hat B_1\check K_1(\hat P^{(k)},P^{(k)})+\hat B_2\hat L\bigr)^\top Y,  \, \forall \,Y \in \mathbb{S}^n,\\
\end{aligned} 
\end{equation}
\begin{equation}
\label{linear_operators2_hk&+1}
\begin{aligned}
&\hat{\mathcal{L}}^{(k,k+1)}(Y) =
Y\bigl(\hat A_{(k)}+\hat B_2\hat T_{(k+1)}\bigr)+\bigl(\hat A_{(k)}+\hat B_2\hat T_{(k+1)}\bigr)^\top Y,  \, \forall \,Y \in \mathbb{S}^n, 
\end{aligned}
\end{equation}
where
\begin{equation}
\label{iterative:matrices_evolve_1}
\begin{cases}
        A_{(k)} = A +  \begin{bmatrix} B_1 & B_2 \end{bmatrix} \begin{bmatrix} K_1(P^{(k)}) \\ K_2(P^{(k)}) \end{bmatrix} ,\hat{A}_{(k)} = \hat{A} +  \begin{bmatrix} \hat{B}_1 & \hat{B}_2 \end{bmatrix} \begin{bmatrix} \hat{K}_1(\hat{P}^{(k)},P^{(k)}) \\ \hat{K}_2(\hat{P}^{(k)},P^{(k)}) \end{bmatrix},\\
        C_{l,(k)} = C_{l} +  \begin{bmatrix} D_{l1} & D_{l2} \end{bmatrix} \begin{bmatrix} K_1(P^{(k)}) \\ K_2(P^{(k)}) \end{bmatrix}, \hat{C}_{l,(k)} = \hat{C}_{l} +  \begin{bmatrix} \hat{D}_{l1} & \hat{D}_{l2} \end{bmatrix} \begin{bmatrix} \hat{K}_1(\hat{P}^{(k)},P^{(k)}) \\ \hat{K}_2(\hat{P}^{(k)},P^{(k)}) \end{bmatrix},  \quad  1 \leq l \leq r ,\\
        T_{(k+1)}=-R_{22}(P^{(k)}+Z^{(k)})^{-1}N_{2,(k)},\hat T_{(k+1)}=-\hat R_{22}(P^{(k)}+Z^{(k)})^{-1}\hat N_{2,(k)},
\end{cases}
\end{equation}
and
\begin{equation}
\label{iterative:matrices_evolve_2} 
\begin{aligned}
    &\begin{bmatrix} N_{1,(k)} \\ N_{2,(k)} \end{bmatrix}=\begin{bmatrix} B_1^{\top} Z^{(k)} + \sum_{l=1}^{r}D_{l,1}^{\top} Z^{(k)} C_{l,(k)} \\ B_2^{\top} Z^{(k)} + \sum_{l=1}^{r}D_{l,2}^{\top} Z^{(k)} C_{l,(k)}  \end{bmatrix},
    \begin{bmatrix} \hat{N}_{1,(k)} \\ \hat{N}_{2,(k)} \end{bmatrix}=\begin{bmatrix} \hat{B}_1^{\top} \hat{Z}^{(k)} + \sum_{l=1}^{r}\hat{D}_{l,1}^{\top} Z^{(k)} \hat{C}_{l,(k)} \\ \hat{B}_2^{\top} \hat{Z}^{(k)} + \sum_{l=1}^{r}\hat{D}_{l,2}^{\top} Z^{(k)} \hat{C}_{l,(k)}  \end{bmatrix}
\end{aligned}
\end{equation}
with $\begin{bmatrix} K_1(P) \\ K_2(P) \end{bmatrix},\begin{bmatrix} \hat{K}_1(\hat{P},P) \\ \hat{K}_2(\hat{P},P) \end{bmatrix}$ is defined in \eqref{main_results:operators_k&n}.

\begin{theorem}
\label{theorem:coupled_iteration}
Assume the following conditions hold:
\begin{itemize}
    \item $R_{22}\succ0$, $\hat{R}_{22}\succ0$ and there exists a pair $(L,\hat{L})\in\mathcal{A}$.
    \item There exists a set of matrices $\{E_{l,(0)}\}_{0 \leq l \leq r}$ and a matrix $\hat{E}_{(0)}$ satisfying
    \begin{equation}
    \sum_{l=0}^r E_{l,(0)}^\top E_{l,(0)} = Q - S^\top(0)R(0)^{-1}S(0),\\
    \hat{E}_{(0)}^\top \hat{E}_{(0)} = \hat Q - \hat{S}^\top(0)\hat{R}(0)^{-1}\hat{S}(0),
    \end{equation}
    such that the systems
    \[
    \left[ E_0,\{E_l\}_{ 1 \leq l \leq r}; A_{(0)}, \{C_{1,(0)}\}_{ 1 \leq l \leq r} \right] \quad \text{and} \quad [\hat{E}_{(0)};\hat{A}_{(0)}]
    \]
    are both \textit{detectable}.
\end{itemize}
Construct the four sequences $\{Z^{(k)}\},\{P^{(k)}\},\{\hat Z^{(k)}\},\{\hat P^{(k)}\}$ as follows:
\begin{enumerate}
    \item Set $P^{(0)}=0$, $\hat P^{(0)}=0$.  
    Let $Z^{(0)}$ be the stabilizing solution of the ARE
    \begin{equation}
    \label{eq:initial_Z}
    \begin{aligned}
    &Z A_{(0)}+A_{(0)}^\top Z+\sum_{l=1}^r C_{l,(0)}^\top Z C_{l,(0)}
    +\mathcal{G}(P^{(0)}) \\
    &-\Bigl(B_2^\top Z+\sum_{l=1}^r D_{l,2}^\top Z C_{l,(0)}\Bigr)^\top
    \bigl(R_{22}+\sum_{l=1}^r D_{l,2}^\top Z D_{l,2}\bigr)^{-1}
    \Bigl(B_2^\top Z+\sum_{l=1}^r D_{l,2}^\top Z C_{l,(0)}\Bigr)=0,
    \end{aligned}
    \end{equation}
    and define $P^{(1)}=Z^{(0)}$.  
    Thereafter, let $\hat Z^{(0)}$ be the stabilizing solution of
    \begin{equation}
    \label{eq:initial_hatZ}
    \begin{aligned}
    &\hat{Z} \hat{A}_{(0)}+\hat{A}_{(0)}^\top \hat{Z}+\sum_{l=1}^r \hat{C}_{l,(0)}^\top P^{(1)}\hat{C}_{l,(0)}
    +\mathcal{F}(\hat{P}^{(0)},P^{(0)})\\
    &-\Bigl(\hat{B}_2^\top \hat{Z}+\sum_{l=1}^r \hat{D}_{l,2}^\top P^{(1)}\hat{C}_{l,(0)}\Bigr)^\top
    \bigl(\hat{R}_{22}+\sum_{l=1}^r \hat{D}_{l,2}^\top P^{(1)}\hat{D}_{l,2}\bigr)^{-1}
    \Bigl(\hat{B}_2^\top \hat{Z}+\sum_{l=1}^r \hat{D}_{l,2}^\top P^{(1)}\hat{C}_{l,(0)}\Bigr)=0,
    \end{aligned}
    \end{equation}
    and set $\hat{P}^{(1)}=\hat{Z}^{(0)}$.

    \item For $k\ge1$, having obtained $P^{(k)},\hat{P}^{(k)}$ and the auxiliary matrices
    $\mathcal{G}(P^{(k)}),\mathcal{F}(\hat{P}^{(k)},P^{(k)})$ from the previous step, let $Z^{(k)}$ be the stabilizing solution of
    \begin{equation}
    \label{eq:iter_Z}
    \begin{aligned}
    &Z A_{(k)}+A_{(k)}^\top Z+\sum_{l=1}^r C_{l,(k)}^\top Z C_{l,(k)}
    +\mathcal{G}(P^{(k)})\\
    &-\Bigl(B_2^\top Z+\sum_{l=1}^r D_{l,2}^\top Z C_{l,(k)}\Bigr)^\top
    \bigl(R_{22}(P^{(k)})+\sum_{l=1}^r D_{l,2}^\top Z D_{l,2}\bigr)^{-1}
    \Bigl(B_2^\top Z+\sum_{l=1}^r D_{l,2}^\top Z C_{l,(k)}\Bigr)=0,
    \end{aligned}
    \end{equation}
    put $P^{(k+1)}=P^{(k)}+Z^{(k)}$, and then let $\hat{Z}^{(k)}$ be the stabilizing solution of
    \begin{equation}
    \label{eq:iter_hatZ}
    \begin{aligned}
    &\hat{Z} \hat{A}_{(k)}+\hat{A}_{(k)}^\top \hat{Z}+\sum_{l=1}^r \hat{C}_{l,(k)}^\top P^{(k+1)}\hat{C}_{l,(k)}+\mathcal{F}(\hat{P}^{(k)},P^{(k)}) \\
    &-\Bigl(\hat{B}_2^\top \hat{Z}+\sum_{l=1}^r \hat{D}_{l,2}^\top P^{(k+1)}\hat{C}_{l,(k)}\Bigr)^\top
    \bigl(\hat{R}_{22}+\sum_{l=1}^r \hat{D}_{l,2}^\top P^{(k+1)}\hat{D}_{l,2}\bigr)^{-1}
    \Bigl(\hat{B}_2^\top \hat{Z}+\sum_{l=1}^r \hat{D}_{l,2}^\top P^{(k+1)}\hat{C}_{l,(k)}\Bigr)=0,
    \end{aligned}
    \end{equation}
    and set $\hat{P}^{(k+1)}=\hat{P}^{(k)}+\hat{Z}^{(k)}$.
\end{enumerate}
Then, for every $k=0,1,2,\dots$ the following properties hold jointly:
\begin{enumerate}
    \item[(a$_k$)] $\operatorname{Spec} \,\mathcal{L}^{(k,k+1)}\subset\mathbb{C}^{-}$ \emph{and} $\operatorname{Spec} \,\hat{\mathcal{L}}^{(k,k+1)}\subset\mathbb{C}^{-}$;
    \item[(b$_k$)] For any $(L,\hat{L})\in\mathcal{A}$,
    $\tilde P_L\succeq P^{(k)}+Z^{(k)}$,
    $\hat P_{L}\succeq \hat P^{(k)}+\hat Z^{(k)}$,
    $\operatorname{Spec} \,\mathcal{L}^*_{P^{(k+1)}}\subset\mathbb{C}^{-}$,
    $\operatorname{Spec} \,\hat{\mathcal{L}}^*_{(\hat P^{(k+1)},P^{(k+1)})}\subset\mathbb{C}^{-}$,
    where $(\tilde P_L,\hat{P}_{L})$ are the stabilizing solutions of the coupled ARE system~\eqref{mflqsdg:coupledARE_l};
    \item[(c$_k$)] $P^{(k)},P^{(k)}+Z^{(k)}\in\mathrm{Dom}\,\mathcal{G}$,
    $(\hat{P}^{(k)},P^{(k)}),(\hat{P}^{(k)}+\hat{Z}^{(k)},P^{(k)}+Z^{(k)})\in\mathrm{Dom}\,\mathcal{F}$, and
    \begin{align*}
    \mathcal{G}(P^{(k)}+Z^{(k)})&=
    -\bigl(N_{1,(k)}-R_{12}(P^{(k)}+Z^{(k)})R_{22}(P^{(k)}+Z^{(k)})^{-1}N_{2,(k)}\bigr)^\top \\
    &\times \mathbb{R}_{22}^{\sharp}(P^{(k)}+Z^{(k)})^{-1}
    \bigl(N_{1,(k)}-R_{12}(P^{(k)}+Z^{(k)})R_{22}(P^{(k)}+Z^{(k)})^{-1}N_{2,(k)}\bigr),\\
    \mathcal{F}(\hat{P}^{(k)}+\hat{Z}^{(k)},P^{(k)}+Z^{(k)})&=
    -\bigl(\hat{N}_{1,(k)}-\hat{R}_{12}(P^{(k)}+Z^{(k)})\hat{R}_{22}^{-1}(P^{(k)}+Z^{(k)})\hat{N}_{2,(k)}\bigr)^\top \\
    &\times \hat{\mathbb{R}}_{22}^{\sharp}(P^{(k)}+Z^{(k)})^{-1}
    \bigl(\hat{N}_{1,(k)}-\hat{R}_{12}(P^{(k)}+Z^{(k)})\hat{R}_{22}^{-1}(P^{(k)}+Z^{(k)})\hat{N}_{2,(k)}\bigr),
    \end{align*}
\end{enumerate}
where $\mathbb{R}_{22}^{\sharp}(P)=R_{11}(P)-R_{12}(P)R_{22}(P)^{-1}R_{21}(P),\hat{\mathbb{R}}_{22}^{\sharp}(P)=\hat{R}_{11}(P)-\hat{R}_{12}(P)\hat{R}_{22}(P)^{-1}\hat{R}_{21}(P)$, $N_{1,(k)},\hat{N}_{1,(k)},N_{2,(k)}$, and $\hat{N}_{2,(k)}$ are defined in \eqref{iterative:matrices_evolve_2}.

Moreover, the sequences converge monotonically:
\[
\lim_{k\to\infty}P^{(k)}=\tilde P,\qquad
\lim_{k\to\infty}\hat P^{(k)}=\hat{P},
\]
where $(\tilde P,\hat{P})$ are the stabilizing solutions of the stochastic coupled GTARE \eqref{mflqsdg:gtare}.
\end{theorem}

\begin{proof}
\textbf{Step 1. Base case ($k=0$).}
Let $P^{(0)} = 0$ and $\hat{P}^{(0)} = 0$. We first establish the well-posedness of the stabilizing solutions $(Z^{(0)},\hat{Z}^{(0)})$ to the initial coupled ARE \eqref{eq:initial_Z} and \eqref{eq:initial_hatZ}.

Since $\mathcal{A} \neq \emptyset$, fix an arbitrary $(L,\hat{L}) \in \mathcal{A}$ and set $(L^{(0)},\hat{L}^{(0)} = (L,\hat{L})$. By definition of the coupled stabilizability set $\mathcal{A}$, using Proposition 1 in \cite{Li2025},  the system 
\[
[A_{(0)}, \bar{A}_{(0)}, \{C_{l,(0)}\}_{1 \leq l \leq r}, \{\bar{C}_{l,(0)}\}_{1 \leq l \leq r}; B_{1}, \bar{B}_{1}, \{D_{l,(0)}\}_{1 \leq l \leq r}, \{\bar{D}_{l,(0)}\}_{1 \leq l \leq r}]
\] 
is \textit{MF-$L^2$-stabilizable}.

Combined with the assumptions $R_{22} \succ 0$, $\hat{R}_{22} \succ 0$,
\[
\sum_{l=0}^{r}E^{\top}_{l,(0)}E_{l,(0)} = Q - S^\top(0) R(0)^{-1}S(0) \succeq 0,
\hat{E}^{\top}_{(0)}\hat{E}_{(0)}\\ = \hat{Q} - \hat{S}^\top(0) \hat{R}(0)^{-1}\hat{S}(0) \succeq 0,
\]
and the detectability of $\left[ E_0,\{E_l\}_{ 1 \leq l \leq r}; A_{(0)}, \{C_{1,(0)}\}_{ 1 \leq l \leq r} \right]$ and $[\hat{E}_{(0)}; \hat{A}_{(0)}]$, Proposition \ref{main_results:coupled_ARE} guarantees the existence and uniqueness of stabilizing solutions $(Z^{(0)},\hat{Z}^{(0)})$ satisfying $R_{22}(Z^{(0)}) \succ 0$ and $\hat{R}_{22}(Z^{(0)}) \succ 0$. Consequently,
$\operatorname{Spec}\, \mathcal{L}^{(0,1)} \subset \mathbb{C}^{-},\operatorname{Spec}\, \hat{\mathcal{L}}^{(0,1)} \subset \mathbb{C}^{-}$.

Let $(\tilde{P}_{L},\hat{P}_{L})$ stand for the stabilizing solution of the coupled ARE \eqref{mflqsdg:coupledARE_l} corresponding to $(L,\hat{L})$. By construction, the spectral inclusions $\mathcal{L}^*_{P^{(0)}} \subset \mathbb{C}^{-}$ and $\hat{\mathcal{L}}^*_{(\hat{P}^{(0)}, P^{(0)})} \subset \mathbb{C}^{-}$ are satisfied, whence Lemma \ref{lemma:coupled_joint_compare} implies
\[
\tilde{P}_L \succeq P^{(0)} + Z^{(0)}, \quad \hat{P}_{L} \succeq \hat{P}^{(0)} + \hat{Z}^{(0)},
\]
and
\[
\operatorname{Spec}\, \mathcal{L}^*_{P^{(1)}} \subset \mathbb{C}^{-}, \quad \operatorname{Spec}\, \hat{\mathcal{L}}^*_{(\hat{P}^{(1)}, P^{(1)})} \subset \mathbb{C}^{-}.
\]

From $R_{22}(\tilde{P}_{L}) \succeq R_{22} \succ 0$, $\hat{R}_{22}(\tilde{P}_{L}) \succeq \hat{R}_{22} \succ 0$, $R_{11} \preceq R_{11}(\tilde{P}_{L}) \prec 0$, and $\hat{R}_{11} \preceq \hat{R}_{11}(\tilde{P}_{L}) \prec 0$, we have $P^{(0)}, \tilde{P}_{L} \in \mathrm{Dom}\,\mathcal{G}$ and $(\hat{P}^{(0)}, P^{(0)}), (\hat{P}_{L}, \tilde{P}_{L}) \in \mathrm{Dom}\,\mathcal{F}$. Moreover, $R_{22}(P^{(1)}) \succeq R_{22} \succ 0$, $\hat{R}_{22}(P^{(1)}) \succeq \hat{R}_{22} \succ 0$, $R_{11}(P^{(1)}) \prec R_{11}(\tilde{P}_{L}) \prec 0$, and $\hat{R}_{11}(P^{(1)}) \prec \hat{R}_{11}(\tilde{P}_{L}) \prec 0$ ensure $P^{(0)} + Z^{(0)} \in \mathrm{Dom}\,\mathcal{G}$ and $(\hat{P}^{(0)} + \hat{Z}^{(0)}, P^{(0)} + Z^{(0)}) \in \mathrm{Dom}\,\mathcal{F}$.

Substituting \eqref{eq:initial_Z} and \eqref{eq:initial_hatZ} into the expressions for $\mathcal{G}(P^{(0)}+Z^{(0)})$ and $\mathcal{F}(\hat{P}^{(0)}+\hat{Z}^{(0)}, P^{(0)}+Z^{(0)})$ and applying Proposition \ref{function_gf:proposition_3}, we obtain the desired formulas in property $c_0$. This completes the proof of $a_0$--$c_0$.

\textbf{Step 2. Inductive step.}
Assume that for some $h \geq 1$, $Z^{(h-1)}$ and $\hat{Z}^{(h-1)}$ are well-defined and properties $a_{h-1}$--$c_{h-1}$ hold. We prove that $Z^{(h)}$ and $\hat{Z}^{(h)}$ are well-defined and $a_h$--$c_h$ hold.

By the inductive hypothesis $c_{h-1}$, setting $E_{l,(h)}\in \mathbb{R}^{[m_1(r+1)]\times n},\hat{E}_{(h)} \in \mathbb{R}^{(n+m_1)\times n}$ for $0 \leq l \leq r$, and
\[
    E_{l,(h)}=\begin{bmatrix} \mathbb{O}_{m_1 \times n}^{\times (l)} \\ \sqrt{-\frac{1}{r+1}\mathbb{R}_{22}^\sharp\left( P^{(h)} \right)^{-1} }\left(N_{1,(h-1)}-R_{12}(P^{(h)})R_{22}(P^{(h)})^{-1}N_{2,(h-1)}\right) \\ \mathbb{O}_{m_1 \times n}^{\times (r-l)}\end{bmatrix},
\]
\[
\hat{E}_{(h)}=\begin{bmatrix}
\sqrt{\sum_{l=1}^r \hat{C}_{l,(h)}^\top P^{(h+1)}\hat{C}_{l,(h)}}\\
    \sqrt{-\hat{\mathbb R}_{22}^{\sharp}(P^{(h)})^{-1}}
    \bigl(\hat N_{1,(h-1)}-\hat R_{12}(P^{(h)}\hat R_{22}^{-1}(P^{(h)})\hat N_{2,(h-1)}\bigr)
\end{bmatrix},
\]
we have
\[
\mathcal{G}(P^{(h)}) = \sum_{l=0}^{r}E^{\top}_{l,(h)}E_{l,(h)} \succeq 0, \quad
\sum_{l=1}^r \hat{C}_{l,(h)}^\top P^{(h+1)}\hat{C}_{l,(h)}+\mathcal{F}(\hat{P}^{(h)}, P^{(h)}) = \hat{E}_{(h)}^\top \hat{E}_{(h)} \succeq 0.
\]

For $k=h$, $Z^{(h)}$ and $\hat{Z}^{(h)}$ satisfy the coupled AREs \eqref{eq:iter_Z} and \eqref{eq:iter_hatZ}. By Proposition \ref{main_results:coupled_ARE}, their well-posedness is equivalent to verifying the \textit{MF-$L^2$-stabilizable} and detectability of the associated systems.

Define $L^{(h)} = L - \check{K}_2(P^{(h)})$ and $\hat{L}^{(h)} = \hat{L} - \check{K}_2(\hat{P}^{(h)}, P^{(h)})$. The systems 
\[
[A_{(h)}+B_2L^{(h)}, \{C_{l,(h)}+D_{l,2}L^{(h)}\}_{1 \leq l \leq r}]
\]
and
$
[\hat{A}_{(h)}+\hat{B}_2\hat{L}^{(h)}]
$
are associated with $\mathcal{L}^*_{P^{(h)}}$ and $\hat{\mathcal{L}}^*_{(\hat{P}^{(h)}, P^{(h)})}$, respectively. By inductive hypothesis $b_{h-1}$, both operators have spectra contained in $\mathbb{C}^{-}$, hence the required \textit{MF-$L^2$-stabilizable} for system 
\[
[A_{(h)}, \bar{A}_{(h)}, \{C_{l,(h)}\}_{1 \leq l \leq r}, \{\bar{C}_{l,(h)}\}_{1 \leq l \leq r}; B_{1}, \bar{B}_{1}, \{D_{l,(h)}\}_{1 \leq l \leq r}, \{\bar{D}_{l,(h)}\}_{1 \leq l \leq r}]
\]
follows.

For detectability, define auxiliary matrices $\tilde{\varTheta}^{(h)}$ and $\hat{\varTheta}^{(h)}$ as 
\[
\tilde{\varTheta}^{(h)} =\begin{bmatrix} \tilde{\varTheta}_{0}^{(h)} & \dotsb  & \tilde{\varTheta}_{r}^{(h)}\end{bmatrix}\in \mathbb{R}^{n\times [m_1(r+1)]},\hat{\varTheta}^{(h)} =\begin{bmatrix} \mathbb{O}_{n \times n} &  \hat{\varTheta}^{(h)}\end{bmatrix}\in \mathbb{R}^{n\times (n+m_1)},
\]
with
\[
\tilde{\varTheta}_{0}^{(h)}=-\left(B_{1}-B_{2}R_{22}(P^{(h)} )^{-1}R_{21}(P^{(h)} )\right)\sqrt{-(r+1)\mathbb{R}_{22}^\sharp\left( P^{(h)} \right)^{-1} },
\]
\[
\tilde{\varTheta}_{l}^{(h)}=  -\left(D_{l1}-D_{l2}R_{22}(P^{(h)})^{-1}R_{21}(P^{(h)} )\right)\sqrt{-(r+1)\mathbb{R}_{22}^\sharp\left( P^{(h)}  \right)^{-1} },\quad 1 \leq l \leq r,
\]
\[
\hat{\varTheta}^{(h)}=-\left(\hat{B}_{1}-\hat{B}_{2}\hat{R}_{22}(P^{(h)} )^{-1}\hat{R}_{21}(P^{(h)} )\right)\sqrt{-\hat{\mathbb{R}}_{22}^\sharp\left( P^{(h)} \right)^{-1} }.
\]
The systems
\[
[A_{(h)}+\tilde{\varTheta}^{(h)}E_{0,(h)}, \{C_{l,(h)}+\tilde{\varTheta}^{(h)}E_{l,(h)}\}_{1 \leq l \leq r}]
\]
and
\[
[\hat{A}_{(h)}+\hat{\varTheta}^{(h)}\hat{E}_{(h)}]
\]
are associated with $\mathcal{L}^{(h-1,h)}$ and $\hat{\mathcal{L}}^{(h-1,h)}$, respectively. By inductive hypothesis $a_{h-1}$, both systems are stable, which guarantees the detectability of
$\left[ [E_{0,(h)},\{E_{l,(h)}\}_{ 1 \leq l \leq r}; A_{(h)}, \{C_{1,(h)}\}_{ 1 \leq l \leq r} \right]$ and $[\hat{E}_{(h)}; \hat{A}_{(h)}]$.

Invoking Proposition \ref{main_results:coupled_ARE}, AREs \eqref{eq:iter_Z} and \eqref{eq:iter_hatZ} admit unique stabilizing solutions $Z^{(h)}$ and $\hat{Z}^{(h)}$ such that $R_{22}(P^{(h+1)}) \succ 0$ and $\hat{R}_{22}(P^{(h+1)}) \succ 0$, and
$
\operatorname{Spec}\, \mathcal{L}^{(h,h+1)} \subset \mathbb{C}^{-},\operatorname{Spec}\, \hat{\mathcal{L}}^{(h,h+1)} \subset \mathbb{C}^{-},
$
which establishes property $a_h$.

Applying Lemma \ref{lemma:coupled_joint_compare} to the stable operators $\mathcal{L}^*_{P^{(h)}}$ and $\hat{\mathcal{L}}^*_{(\hat{P}^{(h)}, P^{(h)})}$ in $a_{h-1}$, we obtain
\[
\tilde{P}_L \succeq P^{(h)} + Z^{(h)}, \quad \hat{P}_{L} \succeq \hat{P}^{(h)} + \hat{Z}^{(h)},
\]
and
\[
\operatorname{Spec}\, \mathcal{L}^*_{P^{(h+1)}} \subset \mathbb{C}^{-}, \quad \operatorname{Spec}\, \hat{\mathcal{L}}^*_{(\hat{P}^{(h+1)}, P^{(h+1)})} \subset \mathbb{C}^{-},
\]
which establishes property $b_h$.

From $R_{22}(P^{(h+1)}) \succ 0$, $\hat{R}_{22}(P^{(h+1)}) \succ 0$, $R_{11}(P^{(h+1)}) \preceq R_{11}(\tilde{P}_L) \prec 0$, and $\hat{R}_{11}(P^{(h+1)}) \preceq \hat{R}_{11}(\tilde{P}_L) \prec 0$, we conclude that $P^{(h+1)} \in \mathrm{Dom}\,\mathcal{G}$ and $(\hat{P}^{(h+1)}, P^{(h+1)}) \in \mathrm{Dom}\,\mathcal{F}$. Substituting the iterative AREs into the expressions for $\mathcal{G}(P^{(h+1)})$ and $\mathcal{F}(\hat{P}^{(h+1)}, P^{(h+1)})$ and applying Proposition \ref{function_gf:proposition_3} yields the desired formulas in property $c_h$.

By mathematical induction, properties $a_k$--$c_k$ hold for all $k \in \mathbb{N}$.

\textbf{Step 3. Convergence analysis.}
Accordingly, $\{P^{(k)}\}$ and $\{\hat{P}^{(k)}\}$ form monotonically non-decreasing sequences of symmetric matrices that are uniformly bounded from above by $\tilde{P}_L$ and $\hat{P}_L$, respectively. Applying the monotone convergence theorem for symmetric matrix sequences, the limits
\[
P^* = \lim_{k \to \infty} P^{(k)}, \quad \hat{P}^* = \lim_{k \to \infty} \hat{P}^{(k)}
\]
are well-defined and satisfy $P^* \preceq \tilde{P}_L$, $\hat{P}^* \preceq \hat{P}_L$ for all $(L,\hat{L}) \in \mathcal{A}$. In addition, the following four limit identities hold:
\[
\lim_{h\rightarrow\infty}\sum_{s=h+1}^{\infty}Z^{(s)}=0;\quad\lim_{h\rightarrow\infty}\sum_{s=h+1}^{\infty}\hat{Z}^{(s)}=0;\quad\lim_{k\rightarrow\infty}Z^{(k)}=0;\quad\lim_{k\rightarrow\infty}\hat{Z}^{(k)}=0.
\]

As established in Theorem 3.9 of \cite{wang2025}, $P^*$ constitutes the stabilizing solution to the first constituent equation of the coupled GTARE \eqref{mflqsdg:gtare}. It remains to verify that the pair $(P^*,\hat{P}^*)$ serves as a stabilizing solution for the full coupled GTARE \eqref{mflqsdg:gtare}.

Owing to continuity and the recursive construction of the iterative scheme, the operator equality below holds for every integer $h\geq0$:
\begin{equation*}
    \begin{aligned}
    \mathcal{F}(\hat{P}^*,P^*)
    =&\sum_{s=h+1}^{\infty}\hat{Z}^{(s)}\hat{A}_{(h)} + \sum_{s=h+1}^{\infty}\hat{A}_{(h)}^{\top} \hat{Z}^{(s)} + \sum_{s=h+1}^{\infty}\sum_{l=1}^{r}\hat{C}_{l,(h)}^{\top} Z^{(s)} \hat{C}_{l,(h)}\\
    &- \left( \sum_{s=h+1}^{\infty}\hat{B}_2^{\top} \hat{Z}^{(s)} + \sum_{s=h+1}^{\infty}\sum_{l=1}^{r}\hat{D}_{l,2}^{\top} Z^{(s)}\hat{C}_{l,(h)} \right)^{\top} \left( \hat{R}_{22} + \sum_{l=1}^{r}\hat{D}_{l,2}^{\top} P^* \hat{D}_{l,2} \right)^{-1} \\
    &\quad \times\left( \sum_{s=h+1}^{\infty}\hat{B}_2^{\top} \hat{Z}^{(s)} + \sum_{s=h+1}^{\infty}\sum_{l=1}^{r}\hat{D}_{l,2}^{\top} Z^{(s)}\hat{C}_{l,(h)} \right)+\mathcal{F}\left(\sum_{s=0}^{h}\hat{Z}^{(s)},\sum_{s=0}^{h}Z^{(s)}\right).
    \end{aligned}
\end{equation*}
Combined with the vanishing tail-series limits
\[
\lim_{h\rightarrow\infty}\sum_{s=h+1}^{\infty}Z^{(s)}=0,\qquad \lim_{h\rightarrow\infty}\sum_{s=h+1}^{\infty}\hat{Z}^{(s)}=0,
\]
as well as the relation
\begin{equation*}
    \begin{aligned}
\lim_{h\rightarrow\infty}\mathcal{F}\left(\sum_{s=0}^{h}\hat{Z}^{(s)},\sum_{s=0}^{h}Z^{(s)}\right)
    =&\lim_{h\rightarrow\infty}-\left(\hat{N}_{1,(h)}-\hat{R}_{12}(P^{(h+1)})\hat{R}_{22}(P^{(h+1)})^{-1}\hat{N}_{2,(h)}\right)^{\top}\hat{R}^{\sharp}_{22}(P^{(h+1)})^{-1}\\
    &\times \left(\hat{N}_{1,(h)}-\hat{R}_{12}(P^{(h+1)})\hat{R}_{22}(P^{(h+1)})^{-1}\hat{N}_{2,(h)}\right)=0,
    \end{aligned}
\end{equation*}
we therefore conclude $\mathcal{F}(\hat{P}^*,P^*)=0$, which confirms that $(P^*,\hat{P}^*)$ is the solution to the full coupled GTARE \eqref{mflqsdg:gtare}.

Suppose $(P^*,\check{P})$ denotes an arbitrary stabilizing solution of the coupled GTARE \eqref{mflqsdg:gtare}. Define $L\triangleq \check{K}_2(P^*)$ and $\check{L}\triangleq \check{K}_2(\check{P},P^*)$, where the operator mappings $\check{K}_2(\cdot),\check{K}_2(\cdot,\cdot)$ are specified in \eqref{iterative:matrices_evolve0}. By Proposition \ref{function_gf:proposition_2}, the coupled GTARE \eqref{mflqsdg:gtare} can be recast into the coupled ARE \eqref{mflqsdg:coupledARE_l} parameterized by $(L,\check{L})$, with $(P^*,\check{P})$ being its corresponding stabilizing solution. From detectability reasoning parallel to Step 2 together with Lemma \ref{main_results:stochastically_detectable_lemma}, we obtain $(L,\check{L})\in\mathcal{A}$.

Invoking Proposition \ref{function_gf:proposition_3}, we split $\hat{P}^* = (\hat{P}^* - \check{P}) + \check{P}$ and expand $\mathcal{F}(\hat{P}^*,P^*)$ into the following form:
\begin{align*}
\mathcal{F}(\hat{P}^*,P^*)
=&(\hat{P}^*-\check{P})\left(\hat{A} + \hat{B}_1\hat{K}_1(\check{P},P^*)+\hat{B}_2\hat{K}_2(\check{P},P^*)\right)+\\
&\left(\hat{A} + \hat{B}_1\hat{K}_1(\check{P},P^*)+\hat{B}_2\hat{K}_2(\check{P},P^*)\right)^{\top}(\hat{P}^*-\check{P})\\
&+\sum_{l=1}^{r}\left(\hat{C}_{l} + \hat{D}_{l,1}\hat{K}_1(\check{P},P^*)+ \hat{D}_{l,2}\hat{K}_1(\check{P},P^*)\right)^{\top}P^*\left(\hat{C}_{l} + \hat{D}_{l,1}\hat{K}_1(\check{P},P^*)+ \hat{D}_{l,2}\hat{K}_1(\check{P},P^*)\right)\\
&-\hat{N}_2(\check{P},P^*,\hat{P}^*-\check{P},0)^{\top} \hat{R}_{22}(P^*)^{-1} \hat{N}_2(\check{P},P^*,\hat{P}^*-\check{P},0)-M^{\top}\mathbb{\hat{R}}_{22}^{\sharp}(P^*)^{-1}M+\mathcal{F}(\check{P},P^*),
\end{align*}
in which
\[
M =  \hat{N}_1(\check{P},P^*,\hat{P}^*-\check{P},0)-\hat{R}_{12}(P^*)\hat{R}_{22}(P^*)^{-1}\hat{N}_2(\check{P},P^*,\hat{P}^*-\check{P},0),
\]
and $\hat{N}_1,\hat{N}_2$ are defined in \eqref{main_results:operators_k&n}. Clearly, the quadratic term satisfies $-M^{\top}\mathbb{\hat{R}}_{22}^{\sharp}(P^*)^{-1}M \succeq 0$.

By virtue of Lemma \ref{main_results:stochastically_detectable_lemma}, the matrix semidefinite ordering $\hat{P}^* \succeq \check{P}$ holds. Recalling the relation $\hat{P}_{\check{L}} = \check{P}$, we obtain the sandwich inequality $\check{P} \preceq \hat{P}^* \preceq \hat{P}_{\check{L}} = \check{P}$. This immediately forces $\hat{P}^*= \check{P}$; in other words, the constructed iterative sequence converges to the unique stabilizing solution of the stochastic coupled GTARE \eqref{mflqsdg:gtare}.

This completes the proof of the theorem.
\end{proof}

\section{Numerical Example}
\label{sec:example}

To verify the effectiveness of the proposed computational method, we present a numerical example as follows:

\begin{align*}
A &= \begin{bmatrix}
        0.408487 & 0.408316 & -0.613031 \\
        -1.051924 & -0.418930 & 0.752468 \\
        -1.691520 & -0.563314 & -0.638052
\end{bmatrix} &
Q &= \begin{bmatrix}
        2.078924 & 0.103475 & 0.105020 \\
        0.103475 & 2.199993 & 0.235006 \\
        0.105020 & 0.235006 & 2.145319
\end{bmatrix}
\end{align*}

\begin{align*}
\bar{A} &= \begin{bmatrix}
        -0.795393 & -0.164484 & 0.133485 \\
        0.724888 & -1.290015 & 0.245408 \\
        -0.413610 & 0.070707 & -0.307858
\end{bmatrix} &
\bar{Q} &= \begin{bmatrix}
        1.941410 & 0.113238 & 0.030321 \\
        0.113238 & 2.279649 & 0.135375 \\
        0.030321 & 0.135375 & 2.079688
\end{bmatrix}
\end{align*}

\begin{align*}
B_1 &= \begin{bmatrix}
        2.967512 & -0.377970 & -0.185514 \\
        -0.018924 & 2.887587 & -0.258110 \\
        -0.388666 & -0.106649 & 2.665765
\end{bmatrix} &
B_2 &= \begin{bmatrix}
        7.307403 & 0.264639 & 0.380914 \\
        0.443168 & 7.430662 & 0.098221 \\
        0.334017 & 0.390885 & 7.005462
\end{bmatrix}
\end{align*}

\begin{align*}
\bar{B}_{1} &= \begin{bmatrix}
        2.587948 & -0.300921 & -0.422051 \\
        -0.051107 & 2.554979 & -0.046939 \\
        -0.429258 & -0.399601 & 2.506959
\end{bmatrix} &
\bar{B}_{2} &= \begin{bmatrix}
        7.484557 & 0.302974 & 0.003160 \\
        0.044090 & 7.149427 & 0.208385 \\
        0.037259 & 0.058320 & 7.185435
\end{bmatrix}
\end{align*}

\begin{align*}
C_1 &= \begin{bmatrix}
        1.104243 & 2.717840 & -0.633231 \\
        -2.267561 & 1.476320 & -0.288744 \\
        -0.259802 & -2.449311 & -1.479046
\end{bmatrix} &
C_2 &= \begin{bmatrix}
        1.607085 & 0.824551 & -0.919251 \\
        -0.932793 & -0.057529 & -0.580398 \\
        -0.715109 & 1.091784 & -1.252092
\end{bmatrix}
\end{align*}

\begin{align*}
\bar{C}_{1} &= \begin{bmatrix}
        -1.080545 & 0.847506 & 1.133466 \\
        -0.995120 & -1.946739 & 2.311947 \\
        0.086116 & -0.495132 & 0.741186
\end{bmatrix} &
\bar{C}_{2} &= \begin{bmatrix}
        -0.062782 & -0.589701 & 1.014898 \\
        0.167362 & -1.128953 & 0.224506 \\
        0.722742 & -0.222852 & -0.314263
\end{bmatrix}
\end{align*}

\begin{align*}
D_{11} &= \begin{bmatrix}
        0.005556 & 0.009032 & 0.000095 \\
        0.005498 & 0.002891 & 0.009092 \\
        0.002450 & 0.001407 & 0.001658
\end{bmatrix} &
D_{12} &= \begin{bmatrix}
        0.000544 & 0.000258 & 0.007273 \\
        0.004977 & 0.004754 & 0.000783 \\
        0.009926 & 0.005598 & 0.000617
\end{bmatrix}
\end{align*}

\begin{align*}
D_{21} &= \begin{bmatrix}
        0.005964 & 0.002042 & 0.000746 \\
        0.004253 & 0.003787 & 0.003855 \\
        0.009422 & 0.006086 & 0.001819
\end{bmatrix} &
D_{22} &= \begin{bmatrix}
        0.004027 & 0.007070 & 0.006877 \\
        0.004564 & 0.002345 & 0.005071 \\
        0.007355 & 0.007320 & 0.002065
\end{bmatrix}
\end{align*}

\begin{align*}
\bar{D}_{11} &= \begin{bmatrix}
        0.003863 & 0.002066 & 0.004529 \\
        0.001781 & 0.005175 & 0.002537 \\
        0.009128 & 0.006056 & 0.007500
\end{bmatrix} &
\bar{D}_{12} &= \begin{bmatrix}
        0.000101 & 0.008952 & 0.003808 \\
        0.005151 & 0.009692 & 0.009632 \\
        0.000458 & 0.006072 & 0.004141
\end{bmatrix}
\end{align*}

\begin{align*}
\bar{D}_{21} &= \begin{bmatrix}
        0.002059 & 0.003393 & 0.005214 \\
        0.005400 & 0.005027 & 0.001769 \\
        0.005689 & 0.001261 & 0.009123
\end{bmatrix} &
\bar{D}_{22} &= \begin{bmatrix}
        0.001866 & 0.003273 & 0.002908 \\
        0.008619 & 0.008379 & 0.004893 \\
        0.007199 & 0.004095 & 0.006787
\end{bmatrix}
\end{align*}

\begin{align*}
R_{11} &= \begin{bmatrix}
        -5.535711 & -1.243871 & -1.060904 \\
        -1.243871 & -5.421049 & -0.883191 \\
        -1.060904 & -0.883191 & -5.179607
\end{bmatrix} &
R_{12} &= \begin{bmatrix}
        0.870715 & 0.214393 & 0.212937 \\
        0.153759 & 0.467172 & 0.877726 \\
        0.701728 & 0.509245 & 0.203460
\end{bmatrix}
\end{align*}

\begin{align*}
R_{21} &= \begin{bmatrix}
        0.870715 & 0.153759 & 0.701728 \\
        0.214393 & 0.467172 & 0.509245 \\
        0.212937 & 0.877726 & 0.203460
\end{bmatrix} &
R_{22} &= \begin{bmatrix}
        5.012442 & 0.080085 & 0.103574 \\
        0.080085 & 5.626602 & 0.572935 \\
        0.103574 & 0.572935 & 6.171741
\end{bmatrix}
\end{align*}

\begin{align*}
\bar{R}_{11} &= \begin{bmatrix}
        -4.125856 & -0.082192 & -0.097008 \\
        -0.082192 & -4.318547 & -0.118512 \\
        -0.097008 & -0.118512 & -4.252670
\end{bmatrix} &
\bar{R}_{12} &= \begin{bmatrix}
        0.148495 & 0.772843 & 0.581107 \\
        0.826663 & 0.667771 & 0.810371 \\
        0.233090 & 0.751656 & 0.122593
\end{bmatrix}
\end{align*}

\begin{align*}
\bar{R}_{21} &= \begin{bmatrix}
        0.148495 & 0.826663 & 0.233090 \\
        0.772843 & 0.667771 & 0.751656 \\
        0.581107 & 0.810371 & 0.122593
\end{bmatrix} &
\bar{R}_{22} &= \begin{bmatrix}
        5.456714 & 0.284852 & 0.850576 \\
        0.284852 & 5.225094 & 0.499638 \\
        0.850576 & 0.499638 & 7.276496
\end{bmatrix}
\end{align*}

\begin{align*}
S_1 &= \begin{bmatrix}
        0.053291 & 0.264513 & 0.938280 \\
        0.098096 & 0.241187 & 0.530500 \\
        0.480865 & 0.854573 & 0.153926
\end{bmatrix} &
S_2 &= \begin{bmatrix}
        0.188821 & 0.738823 & 0.947789 \\
        0.546263 & 0.608000 & 0.326058 \\
        0.592675 & 0.909203 & 0.406563
\end{bmatrix}
\end{align*}

\begin{align*}
\bar{S}_{1} &= \begin{bmatrix}
        0.921735 & 0.562137 & 0.177509 \\
        0.132059 & 0.560981 & 0.962898 \\
        0.996229 & 0.026843 & 0.640519
\end{bmatrix} &
\bar{S}_{2} &= \begin{bmatrix}
        0.372324 & 0.311355 & 0.552001 \\
        0.153874 & 0.828466 & 0.277204 \\
        0.658175 & 0.241821 & 0.146745
\end{bmatrix}
\end{align*}

After 3 external iterations, the iterative sequence converges to the stabilizing solution of the problem within the specified error tolerance. The norm of the equation residual is $1.96213 \times 10^{-9}$ and $1.81781 \times 10^{-8}$. The specific values of the stabilizing solution to the stochastic coupled GTARE \eqref{mflqsdg:gtare} (retained to five decimal places) are presented as follows:
\[
P=\begin{bmatrix}
1.55955 & -0.13514 & -0.30351 \\
-0.13514 & 1.68174 & -0.06873 \\
-0.30351 & -0.06873 & 0.89183
\end{bmatrix},\quad
\hat P=\begin{bmatrix}
1.19285 & 0.02534 & -0.47923 \\
0.02534 & 1.35410 & 0.07357 \\
-0.47923 & 0.07357 & 0.96367
\end{bmatrix}.
\]


\clearpage
\printbibliography

\end{document}